\documentclass[12pt]{amsart}
\usepackage{amsmath,amssymb,amsthm,graphics,amscd,mathrsfs}
\usepackage{amscd, amsfonts, amssymb, amsthm}
\usepackage{mathtools}
\usepackage{fullpage, verbatim}
\usepackage[colorlinks, breaklinks, linkcolor=blue]{hyperref}
\usepackage{array}
\usepackage{latexsym}
\usepackage[shortlabels]{enumitem}
\setlist[enumerate]{label=(\roman*)}
\usepackage[english]{babel}
\usepackage{csquotes}

\usepackage{xcolor}
\usepackage{subcaption}
\usepackage{tikz-cd}
\usepackage{tikz}
\usepackage{cleveref}
\usetikzlibrary{arrows}
\usetikzlibrary{shapes}
\usepackage{tkz-graph}
\usetikzlibrary{decorations}
\usetikzlibrary{arrows,decorations.markings}
\usepackage{multirow}
\tikzstyle{v} = [circle, draw, inner sep=2pt, minimum size=3pt, fill=black]
\tikzstyle{l} = [rectangle, draw, rounded corners]

\usepackage{caption}
\usetikzlibrary{calc}

\newcommand\BBC{{\mathbb C}}
\newcommand\BBH{{\mathbb H}}

\newcommand\BBR{{\mathbb R}}
\newcommand\BBZ{{\mathbb Z}}

\newcommand\Fix{{\operatorname{Fix}}}

\newcommand\GL{\operatorname{GL}}
\newcommand\GU{\operatorname{GU}}

\newcommand\rk{\operatorname{rk}}
\newcommand\SL{\operatorname{SL}}

\newcommand\Sp{\operatorname{Sp}}

\newcommand\diag{{\operatorname{diag}}}
\newcommand\ord{{\operatorname{ord}}}

\newcommand\id{{\operatorname{id}}}

\newcommand\bfi{\mathbf{i}}
\newcommand\bfj{\mathbf{j}}
\newcommand\bfk{\mathbf{k}}
\renewcommand\phi{\varphi}
\renewcommand\epsilon{\varepsilon}

\numberwithin{equation}{section}

\theoremstyle{plain}
\newtheorem{lemma}[equation]{Lemma}
\newtheorem{theorem}[equation]{Theorem}

\newtheorem{prop}[equation]{Proposition}
\theoremstyle{definition}
\newtheorem{defn}[equation]{Definition}
\newtheorem{remark}[equation]{Remark}

\newtheorem{example}[equation]{Example}

\newtheorem{notation}[equation]{Notation}

\subjclass[2010]{}

\begin{document}

\title[On Normalizers of Parabolic Subgroups of Quaternionic Reflection Groups]
{On Normalizers of Parabolic Subgroups of Quaternionic Reflection Groups}

\author[G. Röhrle]{Gerhard Röhrle}
\address
{Fakult\"at f\"ur Mathematik,
Ruhr-Universit\"at Bochum,
D-44780 Bochum, Germany}
\email{gerhard.roehrle@rub.de}

\author[J. Schmitt]{Johannes Schmitt}
\address
{Fakult\"at f\"ur Mathematik,
Ruhr-Universit\"at Bochum,
D-44780 Bochum, Germany}
\email{johannes.schmitt@rub.de}

\allowdisplaybreaks

\begin{abstract}
  By work of Howlett and Muraleedaran--Taylor, a parabolic subgroup of a real or complex reflection group always admits a complement in its normalizer.
  In this note, we investigate this phenomenon for quaternionic reflection groups.
  Here, in contrast to the real and complex setting, we find that complements of parabolic subgroups do not exist in general.
  Indeed, there are infinitely many examples of quaternionic reflection groups in arbitrary rank greater than 2 with a parabolic subgroup that does not admit a complement in its normalizer.
  We give a full classification of parabolic subgroups of irreducible quaternionic reflection groups and describe their complements, if the latter exist.
\end{abstract}

\maketitle

\tableofcontents

\section{Introduction}

For $G$ a finite Coxeter group, the structure of the normalizer $N_G(P)$ of a parabolic subgroup $P$ of $G$ was studied by Howlett \cite{How80} and Brink--Howlett \cite{BH99}.
They in particular show that there is always a complement to $P$ in $N_G(P)$, that is, there is a subgroup \(C\leq N_G(P)\) with \(N_G(P) = P\rtimes C\).
This investigation was extended to complex reflection groups $G$ by Muraleedaran--Taylor \cite{MT18}.
There again a complement to $P$ in $N_G(P)$ always exists.
Owing to Steinberg's theorem, parabolic subgroups of complex reflection groups are themselves complex reflection groups \cite{Ste64}.

Quaternionic reflection groups form a more general class of groups encompassing the class of complex reflection groups \cite{Coh80}.
In this note, we carry out the analogous investigation for the normalizer $N_G(P)$ of a parabolic subgroup $P$ of a quaternionic reflection group $G$.
In \cite[Thm.~1.1]{BST23}, Bellamy, Thiel and the second author showed that the analogue of Steinberg's theorem mentioned above also holds for quaternionic reflection groups.
However, we find that for quaternionic reflection groups a parabolic subgroup need not admit a complement in its normalizer.
Our results may be summarized as follows.

\begin{theorem}
  \label{thm:intro}
  Let \(G\) be an irreducible quaternionic reflection group and let \(P\leq G\) be a parabolic subgroup.
  Then \(P\) has a complement in its normalizer \(N_G(P)\) except if the pair \((G, P)\) is one of the following:
  \begin{enumerate}
    \item \label{thm:intro:1}\(G = G_n(K, H)\) is an imprimitive group with \[(K, H)\in \{(\mathsf D_d, \mathsf C_d), (\mathsf D_d, \mathsf C_{2d}), (\mathsf D_d, \mathsf D_{d/2}), (\mathsf O, \mathsf T)\},\] where \(d\) is even, and \(P = P_\lambda\) with \(\lambda = 1^{b_1}\cdots m^{b_m}\vdash m < n\) and \(\sum_{k\text{ odd}}b_k > n - m\).
    \item \label{thm:intro:2}Five cases among the exceptional groups, namely \[(W(Q), G(4, 2, 2)),\ (W(R), \mathsf C_2),\ (W(S_3), \mathsf C_2),\ (W(U), \mathsf C_2\times \mathsf C_2),\ (W(U), \mathsf C_2\times\mathsf C_2\times\mathsf C_2).\]
  \end{enumerate}
\end{theorem}
The theorem is proved in a case by case fashion in Theorem~\ref{thm:imprim} and Propositions~\ref{prop:dihedralparasrank2}, \ref{prop:imprimrank2excep}, \ref{prop:compp1}, \ref{prop:compp2} and \ref{prop:primprim}.
For the notation used in Theorem~\ref{thm:intro}, see Section~\ref{sec:classify} and the following sections.
Observe that the groups in \ref{thm:intro:1} constitute infinitely many groups \(G\) in arbitrary rank greater or equal 3.

\begin{remark}
  Let \(P\leq G\) be a parabolic subgroup and assume that \(P\) admits a complement \(C\) in the normalizer \(N_G(P)\).
  As for real and complex reflection groups, the group \(C\) acts as a reflection group on the fixed space \(\Fix_V(P)\) in many cases.
  If this is not the case, then \(P\) is among the parabolic subgroups in Theorem~\ref{thm:imprim}~\ref{thm:imprim:2} or \(G\) is one of the exceptional groups, see the tables in Section~\ref{sec:tabsprim}.
\end{remark}

\begin{remark}
  \label{rem:namikawaweyl}
  The description of the normalizers of parabolic subgroups in this article has the following application.
  Given a quaternionic reflection group \(G\leq \GL_n(\BBH)\), there is a faithful embedding of \(G\) into \(\Sp_{2n}(\BBC)\), the symplectic group over \(\BBC\).
  The corresponding linear quotient \(X = \BBC^{2n}/G\) is a \emph{symplectic quotient singularity} \cite{Bea00}.
  As such, \(X\) has a rich birational geometry and one important invariant in this context is the \emph{Namikawa--Weyl group} \cite{Nam10}, which describes isomorphisms between different birational models of \(X\).
  The Namikawa--Weyl group of \(X\) is closely related to the minimal parabolic subgroups of \(G\) and their normalizers, see \cite{Bel16}, \cite{BST18} and the references therein for details.
  Building on this article, we give a complete description of the Namikawa--Weyl groups of symplectic quotient singularities by symplectic reflection groups in \cite{BRS26}.
\end{remark}

The manuscript is organized as follows.
After some preliminaries in Section~\ref{sec:prelim}, we follow the classification of quaternionic reflection groups from \cite{Coh80}.
That is, we consider the imprimitive groups in Section~\ref{sec:imprim}, the primitive groups with imprimitive complexification in Section~\ref{sec:primimprim} and finally the primitive groups with primitive complexification in Section~\ref{sec:primprim}.
To be able to prove Theorem~\ref{thm:intro}, we give complete lists of parabolic subgroups for all irreducible quaternionic reflection groups.

\subsection*{Acknowledgements}
We thank Gwyn Bellamy for pointing out the connection to the Namikawa--Weyl group to us (Remark~\ref{rem:namikawaweyl}).

\section{Preliminaries}
\label{sec:prelim}

\subsection{Parabolic subgroups}
Let \(\BBH\) be the skew-field of quaternions.
We write \(\{1, \bfi, \bfj, \bfk\}\) for the standard basis of \(\BBH\) over \(\BBR\) with \(\BBC = \BBR(\bfi)\).
Let \(V\) be a finite-dimensional right \(\BBH\)-vector space.
Let \(\GL(V)\) be the group of all invertible linear transformations of \(V\).
We agree that \(\GL(V)\) acts on \(V\) from the left.

\begin{defn}
  An element \(g\in \GL(V)\) of finite order is a \emph{quaternionic reflection} (or just \emph{reflection}), if \(\rk(1 - g) = 1\), that is, \(g\) fixes a subspace of codimension 1 in \(V\).
  A finite group \(G\leq\GL(V)\) is a \emph{quaternionic reflection group}, if \(G\) is generated by quaternionic reflections.
\end{defn}

One may consider the quaternionic vector space \(V\) as a complex representation of \(G\) by restriction of scalars.
This gives rise to an operation called \emph{complexification} in \cite{Coh80} which embeds \(G\) into the symplectic group \(\Sp(V|_\BBC)\).
By this point of view, quaternionic reflection groups are also called \emph{symplectic reflection groups}.

\begin{defn}
  Let \(G\leq\GL(V)\) be a quaternionic reflection group. Let $X \subseteq V$. The pointwise stabilizer of $X$ in $G$, $G_X = \{g \in G \mid g \cdot x = x \ \text {for every } x \in X \}$, is called a \emph{parabolic subgroup} of $G$.
\end{defn}

The following is the main result from \cite[Thm.~1.1]{BST23}, generalizing Steinberg's theorem
for complex reflection groups \cite{Ste64} to the setting of quaternionic reflection groups.

\begin{theorem}
  \label{thm:steinberg}
  Let \(G\leq\GL(V)\) be a quaternionic reflection group and let $P$ be a parabolic subgroup of $G$. Then $P$ is itself a quaternionic reflection group, generated by the reflections it contains.
\end{theorem}

\subsection{Reduction to the irreducible case}
We call a quaternionic reflection group \(G\leq\GL(V)\) \emph{(quaternionic) irreducible}, if there is no \(G\)-invariant decomposition \(V = V_1\oplus V_2\) into right \(\BBH\)-vector spaces with \(V_i \neq \{0\}\).
Otherwise, the group \(G\) is called \emph{(quaternionic) reducible}.

Let \(G\) be a group acting reducibly on \(V = V_1\oplus V_2\) as \(G = G_1\times G_2\) with \(G_i\leq\GL(V_i)\) and let \(P\leq G\) be a parabolic subgroup.
Then \cite[Thm.~2.1]{MT18} carries over to the quaternionic case verbatim.
That is, we have \(P = P_1\times P_2\) with \(P_i = (G_i)_{U_i}\) where \(U_i = \Fix_{V_i}(P)\) and \(N_G(P) = N_{G_1}(P_1)\times N_{G_2}(P_2)\).
Further, \(P\) has a complement in \(N_G(P)\) if and only if \(P_1\) and \(P_2\) have complements in \(N_{G_1}(P_1)\) and \(N_{G_2}(P_2)\), respectively.
We may hence restrict our attention to quaternionic irreducible groups.

\subsection{Complex reflection groups}
Let \(G\leq \GL(W)\) be an irreducible complex reflection group for a complex vector space \(W\).
Then we may consider \(G\) as an irreducible quaternionic reflection group acting on \(W\otimes_\BBC\BBH\) by extension of scalars.
The action of \(G\) on \((W\otimes_\BBC\BBH)|_\BBC\) is (complex) reducible and we consequently may call a complex reflection group considered as a quaternionic group a \emph{complex reducible} quaternionic reflection group.

A parabolic subgroup \(P\) of \(G\) as a complex reflection group is the same thing as a parabolic subgroup viewed as a quaternionic reflection group, see also the proof of \cite[Prop.~3.3]{BST23}.
By Muraleedaran--Taylor \cite{MT18}, the group \(P\) has a complement in its normalizer \(N_G(P)\) and this fact is again independent of whether we consider \(G\) as a complex or quaternionic group.
Hence, from now on, we only consider quaternionic reflection groups \(G\leq\GL(V)\) that act irreducibly on \(V|_\BBC\).

\subsection{The classification of quaternionic reflection groups}
\label{sec:classify}
The irreducible quaternionic reflection groups are classified by Cohen \cite{Coh80} with recent amendments by Taylor \cite{Tay25} and Waldron \cite{Wal25}.
We give a short overview of this classification; more details follow in the next sections.

As explained above, the irreducible complex reflection groups as classified by Shephard and Todd \cite{ST54} are naturally included in Cohen's classification.
The remaining irreducible quaternionic reflection groups are divided in \emph{imprimitive} and \emph{primitive} groups.
Here, the group \(G\) is called imprimitive if there is a decomposition \(V = V_1\oplus \cdots\oplus V_k\), \(k \geq 2\), into non-trivial \(\BBH\)-spaces \(V_i\) such that the action of every \(g\in G\) on \(V\) permutes the summands \(V_i\).
The primitive groups do not admit such a decomposition.
The imprimitive groups consist mostly of wreath products \(K\wr S_n\), where \(K\leq \BBH^\times\) is a finite group, and normal subgroups of such products, as a direct generalization of the situation for complex reflection groups.
However, for \(\dim(V) = 2\) there are additional imprimitive groups that do not fit into this pattern.
We consider the imprimitive groups in Section~\ref{sec:imprim}.

The primitive groups are further divided depending on whether \(G\) acts primitively on the complex space \(V|_\BBC\) or not.
We say that \(G\) is a group with \emph{(im)primitive complexification}, respectively.
The primitive groups with imprimitive complexification are of rank at most 2 and there are infinitely many of such; we study these groups in Section~\ref{sec:primimprim}.
The remaining groups with primitive complexification are 16 groups in rank 1 to 5.
These groups are discussed in Section~\ref{sec:primprim}.

\subsection{The Kleinian groups}
Although the study of parabolic subgroups is uninteresting for the reflection groups of rank 1, they are important for what follows as they often appear as building blocks of the groups of higher rank.
The reflection groups of rank 1 are the finite subgroups of \(\BBH^\times\), that is, the well-known Kleinian groups.
We list them here in detail to establish notation and fix generators which we require for computations in the next sections.
\begin{notation}
  \label{not:kleinian}
  Every finite subgroup of \(\BBH^\times\) is \(\BBH^\times\)-conjugate to one of the following.
  \begin{enumerate}
    \item The cyclic groups \(\mathsf C_d = \langle \zeta_d\rangle\) with \(d\geq 1\) and \(\zeta_d\) a primitive \(d\)-th root of unity in \(\BBC^\times\).
    \item The binary dihedral groups \(\mathsf D_d = \langle \zeta_{2d}, \bfj\rangle\) with \(d\geq 2\).
    \item The binary tetrahedral group \(\mathsf T = \langle \mathsf D_2, \omega\rangle\) where \(\omega = \frac{1}{2}(-1 + \bfi + \bfj + \bfk)\in \BBH^\times\) is an element of order 3.
    \item The binary octahedral group \(\mathsf O = \langle \mathsf D_4, \omega\rangle\) with \(\omega\) as above.
    \item The binary icosahedral group \(\mathsf I = \langle \mathsf D_2, \sigma\rangle\) with \(\sigma = \frac{1}{2}(\tau^{-1} + \bfi + \tau \bfj)\in\BBH^\times\) an element of order 5, where \(\tau = \frac{1}{2}(1 + \sqrt{5})\).
  \end{enumerate}
\end{notation}
Note that the quaternion group \(Q_8\) is identical to \(\mathsf D_2\) in this list.
With the above terminology, the cyclic groups are complex reducible, the binary dihedral groups are primitive with imprimitive complexification and the remaining groups are primitive with primitive complexification.

\section{Imprimitive quaternionic reflection groups}
\label{sec:imprim}

Let \(V\) be a finite-dimensional right vector space over \(\BBH\) of dimension \(n\geq 2\) and let \(G\leq \GL(V)\) be a reflection group.
Recall that \(G\) is called \emph{imprimitive} if there is a decomposition \(V = V_1\oplus \cdots\oplus V_k\), \(k \geq 2\), into non-trivial \(\BBH\)-spaces \(V_i\) such that the action of every \(g\in G\) on \(V\) permutes the summands \(V_i\).
By \cite[Thm.~2.9]{Coh80}, the irreducible, imprimitive quaternionic reflection groups of rank at least 3 are given by normal subgroups of certain wreath products.
More precisely, let \(K, H\leq \BBH^\times\) be finite groups with \[[K,K]\leq H\trianglelefteq K,\] where \([K,K]\) denotes the commutator subgroup of \(K\).
Let \[A_n(K, H) := \left\{\left(\begin{smallmatrix}x_1&&\\&\ddots&\\&&x_n\end{smallmatrix}\right)\;\middle|\;x_i\in K,\ x_1\cdots x_n\in H\right\}\leq\GL_n(\BBH).\]
Then every irreducible, imprimitive quaternionic reflection group \(G\) of rank \(n\) is conjugate in \(\GL(V)\) to a group of the form \[G_n(K, H) := A_n(K, H) \rtimes S_n,\] where the symmetric group \(S_n\) acts on an element of \(A_n(K, H)\) by permuting the entries on the diagonal in the natural way.
The group \(G_n(K, H)\) is a normal subgroup of the wreath product \(G_n(K, K) = K\wr S_n\).
To occasionally simplify the notation, we allow \(n = 0\) and denote by \(G_0(K, H)\) the trivial group acting on \(V = \{0\}\).

The options for the pairs \((K, H)\) are as follows \cite[Ch.~20]{DuV64}:
\begin{enumerate}
  \item \(K = \mathsf C_d\) and \(H = \mathsf C_e\) with \(e \mid d\);
  \item \(K = \mathsf D_d\) and \(H \in \{\mathsf C_d, \mathsf C_{2d}, \mathsf D_d\}\); if \(d\) is even, we also have \(H = \mathsf D_{d/2}\);
  \item \(K = \mathsf T\) and \(H\in\{\mathsf D_2, \mathsf T\}\);
  \item \(K = \mathsf O\) and \(H\in\{\mathsf T, \mathsf O\}\);
  \item \(K = \mathsf I\) and \(H = \mathsf I\).
\end{enumerate}
For \(K = \mathsf C_d\) cyclic, the group \(G_n(\mathsf C_d, \mathsf C_e)\) can be identified with a complex reflection group; we have the equality \(G_n(\mathsf C_d, \mathsf C_e) = G(d, d/e, n)\) with the notation from \cite{ST54}.
For \(n = 2\), there are further imprimitive groups in addition to the groups \(G_2(K, H)\); we consider these groups in Section~\ref{sec:imprimexcep}.

\subsection{Complements of parabolic subgroups}
In the following, we fix subgroups \(K, H\leq \BBH^\times\) with \([K, K]\leq H\leq K\) as well as \(n\geq 2\).
Let \(G = G_n(K, H)\).

The parabolic subgroups of \(G\) are described in \cite[Sect.~4]{GRS25}; we summarize the results.
We construct representatives of the conjugacy classes of parabolic subgroups as follows.
Let \(e_1,\dots, e_n\) be the standard basis of \(\BBH^n\).
Let \(m\in\{1,\dots, n\}\) and let \(\lambda = (\lambda_1,\dots,\lambda_k)\) be a partition of \(m\).
Write \(n_0 = n - m\) and let \(\alpha\in K\).
Define \[P_\lambda^\alpha := P_0 \times P_1 \times \cdots \times P_k,\] where \(P_0\) is the quaternionic reflection group \(G_{n_0}(K, H)\) acting on the space spanned by \(\{e_1,\dots, e_{n_0}\}\), \(P_1\) is the group \(S_{\lambda_1}\) permuting the vectors \(\{\alpha e_{n_0 + 1}, e_{n_0 + 2},\dots, e_{n_0 + \lambda_1}\}\) and, for \(2\leq i\leq k\), \(P_i\) is the group \(S_{\lambda_i}\) permuting the vectors \(\{e_j\mid \sum_{k = 1}^{i - 1}\lambda_k < j \leq \sum_{k = 1}^i\lambda_k\}\).
If \(\alpha = 1\), we abbreviate \(P_\lambda := P_\lambda^1\).
By \cite[Thm.~4.12]{GRS25}, any parabolic subgroup of \(G\) is conjugate to either \(P_\lambda\) with \(\lambda \vdash m < n\) or \(P_\lambda^\alpha\) with \(\lambda \vdash n\).

If \(H = \{1\}\) is trivial, then \(K\) is cyclic, because \([K, K]\leq H\) (compare Notation~\ref{not:kleinian}), and then \(G\) is a complex reflection group.
In the next lemma, we may thus assume that \(H\neq \{1\}\).
\begin{lemma}
  \label{lem:ordersimprim}
  Let \(m \leq n\) and \(\lambda = 1^{b_1}2^{b_2}\cdots m^{b_m} \vdash m\).
  Let \(k = b_1 + \cdots + b_m\) and \[b_\lambda := b_1!b_2!\cdots b_m!(1!)^{b_1}(2!)^{b_2}\cdots (m!)^{b_m}.\]
  Assume \(H\neq \{1\}\).
  \begin{enumerate}
    \item\label{lem:ordersimprim:1} For fixed \(\lambda\), the union of the conjugacy classes of the parabolic subgroups of the form \(P^\alpha_\lambda\) contains \[\frac{n!|{K}|^{m - k}}{(n - m)!b_\lambda}\] groups.
    \item\label{lem:ordersimprim:2} For \(P = P_\lambda^\alpha\) and \(N = N_G(P)\) we have
      \begin{align*}
        |{N/P}| = \begin{cases}
          \prod_{i = 1}^m(|{K}|^{b_i}b_i!), & m < n, \\
          \frac{e}{[K:H]}\prod_{i = 1}^m(|{K}|^{b_i}b_i!), & m = n,
        \end{cases}
      \end{align*}
      with
      \begin{align*}
        e = \begin{cases}
          \gcd([K:H],\lambda_1,\dots, \lambda_k), & K/H\text{ is cyclic},\\
          [K:H], & K/H\text{ is not cyclic and }2\mid \gcd(\lambda_1,\dots,\lambda_k),\\
          1, & \text{otherwise,}
          \end{cases}
      \end{align*}
      where \(\lambda = (\lambda_1,\dots, \lambda_k)\).
  \end{enumerate}
\end{lemma}
\begin{proof}
  \begin{enumerate}
    \item This follows as in \cite[Lem.~3.5]{MT18}, \cite[Prop.~6.75]{OT92}.
    \item We have \(|{G}| = |{K}|^{n - 1}|{H}|n!\) and
      \begin{align*}
        |{P}| = \begin{cases}
          |{K}|^{n - m - 1}|{H}|(n - m)!\prod_{i = 1}^k(\lambda_i!), & m < n\\
          \prod_{i = 1}^k(\lambda_i!), & m = n.
        \end{cases}
      \end{align*}
      If \(m = n\), \cite[Thm.~4.12]{GRS25} implies that the groups \(P_\lambda^\alpha\) partition into \(e\) conjugacy classes.
      In all other cases, there is one conjugacy class corresponding to the partition \(\lambda\).
      Together with the length of these classes from part \ref{lem:ordersimprim:1}, this gives the claimed orders of groups.\qedhere
  \end{enumerate}
\end{proof}

For matrices \(A_1,\dots, A_k\), we write \(\diag(A_1,\dots, A_k)\) to denote the block diagonal matrix with the \(A_i\) on the diagonal.
If \(\sigma\in S_k\), \(\diag(A_1,\dots,A_k).\sigma\) denotes the matrix with the corresponding blocks of columns permuted.
We abbreviate \(\diag(g_1,\dots,g_n).\sigma\in G_n(K, H)\) by \[((g_1,\dots, g_n), \sigma)\] with \(g_i\in K\) and \(\sigma\in S_n\).
\begin{notation}
  \label{not:gamma}
  \begin{enumerate}
    \item We define the map \[\delta_n: G_n(K, K) \to K/H,\ ((g_1,\dots,g_n), \sigma)\mapsto g_1\cdots g_n H.\]
      Notice that \(\delta_n\) is a group homomorphism because \(K/H\) is abelian by construction.
    \item Let \(\lambda = 1^{b_1}\cdots m^{b_m}\vdash m \leq n\).
      Set \[\Gamma_\lambda \coloneqq \prod_kG_{b_k}(K, K),\] where \(G_0(K, K)\) is the trivial group as before.
      We have an embedding \[\iota_\lambda:\Gamma_\lambda \to G_m(K, K)\] by setting \[\iota_\lambda((g_1,\dots,g_{b_k}), \tau) = \diag(I_{b_1 + \cdots + (k - 1) b_{k - 1}}, \diag(g_1I_k,\dots,g_{b_k}I_k).\tau, I_{(k + 1)b_{k + 1} + \cdots + mb_m})\] for \(((g_1,\dots,g_{b_k}), \tau)\in G_{b_k}(K, K)\).
    \item We have \(N_G(P_\lambda)\leq G_{n - m}(K, K)\times G_m(K, K)\).
      Let \(\pi_{n - m}: N_G(P_\lambda) \to G_{n - m}(K, K)\) and \(\pi_m: N_G(P_\lambda) \to G_m(K, K)\) be the corresponding projections.
  \end{enumerate}
\end{notation}

\begin{lemma}
  \label{lem:compiso}
  With the above notation, if either \(m < n\) or \(K = H\), there is a surjective map \[\psi:N_G(P_\lambda)\to \Gamma_\lambda\] such that \[\delta_m(\pi_m(x)) = \prod_{k = 1}^m\delta_{b_k}(\psi(x)_k)^k\] for all \(x\in N_G(P_\lambda)\).
  Furthermore, if \(C\leq N_G(P_\lambda)\) is a complement of \(P_\lambda\), then the restriction \(\psi|_C\) is an isomorphism onto \(\Gamma_\lambda\).
\end{lemma}
\begin{proof}
  We have \(P_\lambda = G_{n - m}(K, H)\times \prod_k S_k^{b_k}\) and one checks that \[N_{G_{kb_k}(K, K)}(S_k^{b_k}) = \left\{\!\!\begin{pmatrix} g_1 \sigma_1& & \\ & \ddots & \\ & & g_{b_k}\sigma_{b_k}\end{pmatrix}\!\!.\tau\ \middle|\ g_i\in K, \sigma_i\in S_k, \tau\in S_{b_k}\!\right\},\] where \(\sigma_i\) denotes the corresponding \(k\times k\) permutation matrix and \(\tau\) acts on the matrix by permuting the blocks of \(k\) columns.
  Furthermore, there is a surjective map \[\psi_k: N_{G_{kb_k}(K, K)}(S_k^{b_k}) \to G_{b_k}(K, K),\ \diag(g_1\sigma_1,\dots, g_{b_k}\sigma_{b_k}).\tau \mapsto ((g_1,\dots, g_{b_k}),\tau).\]
  We obtain
  \begin{align}
    \label{eq:compiso:ng}
    N_G(P_\lambda) = \left\{\!\!\begin{pmatrix} A_0 & & \\ & \ddots & \\ & & A_m\end{pmatrix}\ \middle| \begin{array}{l}A_0\in G_{n - m}(K, K), A_k\in N_{G_{kb_k}(K, K)}(S_k^{b_k}) \text{ for } k\geq 1,\\ \delta_{n - m}(A_0) \delta_{b_1}(A_1) \cdots \delta_{mb_m}(A_m) = 1\end{array}\!\!\!\right\},
  \end{align}
  where we assume here and in the following that the entry \(A_0\) is not present if \(n = m\) and the entry \(A_k\) is not present if \(b_k = 0\).
    Because we assume \(m < n\) or \(K = H\), it follows from \eqref{eq:compiso:ng} that the map \[\psi: N_G(P_\lambda)\to \Gamma_\lambda,\ \diag(A_0,\dots, A_m)\mapsto (\psi_1(A_1),\dots, \psi_m(A_m))\] is surjective with kernel \(\ker\psi = P_\lambda\) and that \(\psi\) commutes with the maps \(\delta_k\) as claimed.
  If \(C\leq N_G(P_\lambda)\) is a complement to \(P_\lambda\), so that \(N_G(P_\lambda) = P_\lambda C\), then the restriction \(\psi|_C\) must be surjective as well.
  By Lemma~\ref{lem:ordersimprim}, we conclude that \(\psi|_C\) is an isomorphism.
\end{proof}

\begin{prop}
  \label{prop:compviadiagram}
  Let \(\lambda = 1^{b_1}\cdots m^{b_m}\vdash m < n\).
  The parabolic subgroup \(P_\lambda\) has a complement in \(N_G(P_\lambda)\) if and only if there is a morphism \(\phi:\Gamma_\lambda\to G_{n - m}(K, K)\) such that the diagram
  \begin{equation}
    \label{eq:compviadiagram}
    \begin{tikzcd}
      \Gamma_\lambda\arrow["\phi"]{r} \arrow["\delta_{b_1}\delta_{b_2}^2\cdots \delta_{b_m}^m"']{d}& G_{n - m}(K, K) \arrow["\delta_{n - m}"]{d}\\
      K/H \arrow["\operatorname{inv}"]{r}&K/H
    \end{tikzcd}
  \end{equation}
  commutes, where \(\operatorname{inv}\) is inversion in \(K/H\) (this is a morphism as \(K/H\) is abelian).
\end{prop}
\begin{proof}
  Assume that \(C\leq N_G(P_\lambda)\) is a complement of \(P_\lambda\).
  For every \(x\in C\) we have \(\delta_{n - m}(\pi_{n - m}(x))\delta_m(\pi_m(x)) = 1\).
  By Lemma~\ref{lem:compiso}, there is an isomorphism \(\psi_C:C\to \Gamma_\lambda\) with \(\delta_m(\pi_m(x)) = \prod_k \delta_{b_k}(\psi(x)_k)^k\) for all \(x\in C\).
  Then \(\phi \coloneqq \pi_{n - m}\circ \psi_C^{-1}\) is as desired.

  Assume now that a morphism \(\phi:\Gamma_\lambda\to G_{n - m}(K, K)\) as in the diagram \eqref{eq:compviadiagram} exists.
  Consider the set of matrices in \(G_{n - m}(K, K) \times G_m(K, K)\) defined by \[C_\phi\coloneqq\left\{\!\!\begin{pmatrix} \phi(x) & \\ & \iota_\lambda(x) \end{pmatrix}\ \middle|\ x \in\Gamma_\lambda\right\}\] with the map \(\iota_\lambda\) as in Notation~\ref{not:gamma}.
  By the commutativity of the diagram and the explicit description of \(N_G(P_\lambda)\) in \eqref{eq:compiso:ng}, we have \(C_\phi\subseteq N_G(P_\lambda)\).
  Because \(\phi\) is a homomorphism, \(C_\phi\) is a group.
  Clearly, \(C_\phi\cap P_\lambda = \{1\}\) and \(C_\phi P_\lambda = N_G(P_\lambda)\), so \(C_\phi\) is a complement of \(P_\lambda\).
\end{proof}

Proposition~\ref{prop:compviadiagram} enables us to prove the existence of a complement in many cases.
However, not every parabolic subgroup admits a complement in its normalizer due to the following lemma.

\begin{lemma}
  \label{lem:imprimspecial}
  Let either \(K = \mathsf D_d\) with \(d\) even and \(H\in \{\mathsf C_d, \mathsf C_{2d}, \mathsf D_{d/2}\}\) or \(K = \mathsf O\) and \(H = \mathsf T\).
  Let \(\lambda = 1^{b_1}\cdots m^{b_m}\vdash m < n\).
  Then the parabolic subgroup \(P_\lambda\) of \(G = G_n(K, H)\) has a complement in \(N_G(P_\lambda)\) if and only if \[\sum_{k\text{ odd}}b_k \leq n - m.\]
\end{lemma}
\begin{proof}
  If \(\sum_{k\text{ odd}} b_k \leq n - m\), we may embed the group \(\prod_{k\text{ odd}}G_{b_k}(K, K)\) into \(G_{n - m}(K, K)\) diagonally.
  We may extend this embedding to a map \(\phi: \Gamma_\lambda\to G_{n - m}(K, K)\) by mapping every factor \(G_{b_k}(K, K)\) with \(k\) even to 1.
  For the considered groups, the quotient \(K/H\) is of exponent 2.
  Hence \(\phi\) fits into the diagram in \eqref{eq:compviadiagram} because for \(k\) even we have \(\delta_{b_k}^k(x) = 1\) for all \(x\in G_{b_k}(K, K)\) on the left arrow and the inversion map \(\operatorname{inv}\) is the identity.
  So \(P_\lambda\) has a complement in the normalizer in this case.

  On the other hand, if \(P_\lambda\) does have a complement, then there must be a map \(\phi':\prod_{k\text{ odd}}G_{b_k}(K, K)\to G_{n - m}(K, K)\) such that the diagram
  \[\begin{tikzcd}
    \prod_{k\text{ odd}}G_{b_k}(K, K)\arrow["\phi'"]{rr} \arrow["\prod_{k\text{ odd}}\delta_{b_k}"']{dr}&& G_{n - m}(K, K) \arrow["\delta_{n - m}"]{dl}\\
    & K/H&
  \end{tikzcd}\]
  commutes by the same argument.
  Any such \(\phi'\) restricts to a map \(\phi'|_{K^p}:K^p \to G_{n - m}(K, K)\) that commutes with the maps \(\delta_k\), where \(p = \sum_{k\text{ odd}}b_k\).
  In Lemma~\ref{lem:nosuchmap} below, we prove that the existence of such a map implies \(p \leq n - m\) as claimed.
\end{proof}

\begin{lemma}
  \label{lem:nosuchmap}
  Let \(p, q\geq 1\) and let either \(K = \mathsf D_d\) with \(d\) even and \(H\in \{\mathsf C_d, \mathsf C_{2d}, \mathsf D_{d/2}\}\) or \(K = \mathsf O\) and \(H = \mathsf T\).
  By abuse of notation, we write \(\delta_1^p\) for the map \[\delta_1^p:K^p\to K/H, (x_1,\dots,x_p) \mapsto \delta_1(x_1)\cdots \delta_1(x_p).\]
  Assume there is a homomorphism \(\phi: K^p \to G_q(K, K)\) with \(\delta_q \circ \phi = \delta_1^p\).
  Then \(q\geq p\).
\end{lemma}
\begin{proof}
  Recall that \(\mathsf D_d = \langle \zeta_{2d}, \bfj\rangle\), \(\mathsf T = \langle \mathsf D_2, \omega\rangle\) and \(\mathsf O = \langle \mathsf D_4, \omega\rangle\), where \(\omega\) is an element of order 3.
  Let \((K, H)\) be one of the considered pairs of groups; if \(K = \mathsf O\), let \(d = 4\).
  Let \(\xi = \zeta_{2d}\bfj\) and notice that \(\xi^2 = -1\), \(\xi^4 = 1\) and \(\xi \bfi = \bfi^{-1} \xi\).
  Further, \(\xi\notin H\) and \(\xi^2\in H\).
  In the following, for \(a\in K\), we write \(a\equiv \xi\), if \(aH = \xi H\).
  Assume that a map \(\phi: K^p\to G_q(K, K)\) with \(\delta_q\circ\phi = \delta_1^p\) exists for \(q < p\).
  For \(1\leq i \leq p\), write \(\iota_i : K\to K^p\) for the natural embedding of \(K\) into the \(i\)-th component of \(K^p\) and \(\phi_i = \phi\circ\iota_i\) for the concatenation.

  For some \(i\), let \(\phi_i(\xi) = ((x_1,\dots,x_q), \sigma)\) with \(x_j\in K\) and \(\sigma\in S_q\).
  We have \(\delta_1^p(\phi_i(\xi)) = \xi H\), so \(x_1\cdots x_q \equiv \xi\).
  Hence \(x_j\equiv \xi\) for an odd number of indices \(j\).
  Let \(1\leq j\leq q\) with \(x_j\equiv \xi\) and let \(\Omega_\sigma(j) = \{\sigma^k(j)\mid k \geq 1\}\) be the orbit under \(\sigma\).
  We have \(\phi_i(1) = \phi_i(\xi)^4 = ((1,\dots,1), \id)\), so \(|{\Omega_\sigma(j)}| \leq 4\) and \(x_jx_{\sigma^{-1}(j)}x_{\sigma^{-2}(j)}x_{\sigma^{-3}(j)} = 1\).
  If \(|{\Omega_\sigma(j)}| = 4\), then the number \(|{\{x_{j'}\mid j'\in \Omega_\sigma(j)\text{ and } x_{j'}\equiv \xi\}}|\) must be even.
  If \(|{\Omega_\sigma(j)}| = 2\), then \((x_jx_{\sigma^{-1}(j)})^2 = 1\).
  Assuming \(x_{\sigma^{-1}(j)}\not\equiv \xi\) yields \(x_jx_{\sigma^{-1}(j)} \equiv\xi\), a contradiction.
  So \(x_{\sigma^{-1}(j)}\equiv\xi\) as well.
  Because the total number of \(j\) with \(x_j \equiv \xi\) must be odd, there must hence be \(1\leq j\leq q\) with \(\sigma(j) = j\) and \(x_j\equiv \xi\).
  Choose such an index \(j\).

  Let \(\phi_i(\bfi) = ((y_1,\dots,y_q), \tau)\).
  We have \(\sigma\tau = \tau^{-1}\sigma\) because \(\xi\bfi = \bfi^{-1}\xi\).
  Since \(\phi_i(\bfi)^4 = 1\), we again have \(|{\Omega_\tau(j)}|\leq 4\).
  If \(|{\Omega_\tau(j)}| = 4\), then \(\sigma(\tau^2(j)) = \tau^2(j)\) and \(\sigma(\tau(j)) = \tau^{-1}(j)\), that is, \(\sigma\) has the additional fixed point \(\tau^2(j)\) and acts as transposition on the remaining two elements of \(\Omega_\tau(j)\).
  We have \(\phi_i(\xi)\phi_i(\bfi)^2 = \phi_i(\bfi)^2\phi_i(\xi)\), so in particular \[x_jy_jy_{\tau^{-1}(j)} = y_jy_{\tau^{-1}(j)}x_{\tau^2(j)}.\]
  Hence \(x_{\tau^2(j)}\) is conjugate to \(x_j\) in \(K\), so \(x_{\tau^2(j)}\equiv\xi\).
  Similarly, if \(|{\Omega_\tau(j)}| = 2\), we have \(\sigma(\tau(j)) = \tau(j)\) and \[x_jy_j = y_{\tau(j)}^{-1}x_{\tau(j)}.\]
  Because \((y_jy_{\tau(j)})^2 = 1\), we conclude \(y_jy_{\tau(j)} = \pm1\), so again \(x_{\tau(j)} \equiv\xi\).
  In conclusion, if \(|{\Omega_\tau(j)}| > 1\), then the number \(|{\{x_{j'}\mid j'\in \Omega_\tau(j)\text{ and } x_{j'}\equiv \xi\}}|\) must be even.
  Hence, for every \(1\leq i\leq p\), there must be \(1\leq j \leq q\), such that for \(((x_1,\dots,x_q), \sigma) = \phi_i(\xi)\) and \(((y_1,\dots,y_q), \tau) = \phi_i(\bfi)\), we have \(\sigma(j) = \tau(j) = j\) and \(x_j\equiv\xi\).

  Because we assumed \(q < p\), there must be \(1\leq i \neq i'\leq p\) and \(1\leq j\leq q\), such that in addition to the above notation we have \(((z_1,\dots,z_q), \rho) = \phi_{i'}(\xi)\) with \(\rho(j) = j\) and \(z_j\equiv\xi\).
  Then \(x_jz_j = z_jx_j\) as the elements \(\phi_i(\xi)\) and \(\phi_{i'}(\xi)\) commute.
  One checks that these restrictions on \(x_j\) and \(z_j\) together with \(x_j^4 = z_j^4 = 1\) imply \(x_j = \pm z_j\) for all considered pairs \((K, H)\).
  We have \(x_j y_j = y_j^{-1}x_j\) and \(z_j y_j = y_j z_j\).
  Hence \(y_j = y_j^{-1}\) and this implies \(y_j = \pm 1\).
  Finally, \(\phi_i(\bfi)^2 = \phi_i(\xi)^2\), so \(-1 = x_j^2 = y_j^2\), a contradiction.
\end{proof}

We require a bit more notation to be able to state the main theorem of this section.
\begin{notation}
  Let \(P_\lambda = G_{n - m}(K, H) \times S_{\lambda_1}\times\cdots\times S_{\lambda_k}\leq G_n(K, H)\) with \(\lambda = 1^{b_1}\cdots m^{b_m}\vdash m\leq n\).
  Then we have a direct sum decomposition \(V = V_0 \oplus V_1\oplus \cdots \oplus V_m\) of \(V = \BBH^n\), where \(V_0 = \BBH^{n - m}\) and \(V_k\) is the vector space of dimension \(kb_k\) on which \(P_\lambda\) acts via \(S_k^{b_k}\) for \(k \in\{1,\dots,m\}\).
  Write \(X_k = \Fix_V(P_\lambda)\cap V_k\) for the corresponding fixed spaces.
\end{notation}

\begin{notation}
  Let \(k\geq 1\).
  We denote by \(H^{(k)}\) the largest subgroup \(H'\leq K\) such that \(x^k\in H\) for all \(x\in H'\).
  This is well-defined:
  For \(K = \mathsf C_d\) and \(H = \mathsf C_e\) cyclic, \(H^{(k)} = \mathsf C_f\) with \(f = e \gcd(d/e, k)\).
  In the other cases, we have \(H^{(k)} \in \{K, H\}\), except if \(K = \mathsf D_d\) with \(d\) odd and \(H = \mathsf C_d\) where \(H^{(k)} = \mathsf C_{2d}\) for \(k \equiv 2\) mod 4.
\end{notation}

\begin{notation}
  If a parabolic subgroup \(P\) admits a complement \(C\) in its normalizer \(N_G(P)\), then we denote by \(C^\circ\) the largest subgroup of \(C\) that acts on \(\Fix_V(P)\) as a reflection group.
\end{notation}

\begin{theorem}
  \label{thm:imprim}
  Let \(K, H\leq \BBH^\times\) be finite groups with \([K, K]\leq H\leq K\) and let \(n \geq 2\).
  Let \(G = G_n(K, H)\), \(\lambda = 1^{b_1}\cdots m^{b_m}\vdash m \leq n\) and \(\alpha\in K\).
  \begin{enumerate}
    \item\label{thm:imprim:1} In case \(m < n\), the parabolic subgroup \(P_\lambda\) has a complement \(C\) in \(N_G(P_\lambda)\) if and only if \(\sum_{k\text{ odd}}b_k \leq n - m\) or \((K, H)\notin\{(\mathsf D_d, \mathsf C_d), (\mathsf D_d, \mathsf C_{2d}), (\mathsf D_d, \mathsf D_{d/2}), (\mathsf O, \mathsf T)\}\) with \(d\) even.
      If the complement \(C\) exists, then \(C = C^\circ\) and \(C\) acts as the reflection group \(\prod_k G_{b_k}(K, K)\) on \(\Fix_V(P)\).
    \item\label{thm:imprim:2} In case \(m = n\), the parabolic subgroup \(P_\lambda^\alpha\) has a complement \(C\) in \(N_G(P_\lambda^\alpha)\).
      Further, \(C^\circ\) acts on \(\Fix_V(P)\) as \(\prod_k G_{b_k}(K, H^{(k)})\).
      For all \(k\) with \(b_k\neq 0\), the restriction of \(C\) to \(X_k\) is \(G_{b_k}(K, H')\) with \(H^{(k)}\leq H' \leq K\).
  \end{enumerate}
\end{theorem}
\begin{proof}
  \begin{enumerate}
    \item This is a case by case analysis of the pairs \((K, H)\); most of it follows directly from Proposition~\ref{prop:compviadiagram}.

      \begin{itemize}
        \item If \(K = H\), then \(K/H\) is trivial, so one may choose \(\phi\) with \(\phi(x) = 1\) for all \(x\) in Proposition~\ref{prop:compviadiagram} and the complement exists in all cases.
        \item If \((K, H)\) is one of \((\mathsf C_d, 1)\) with \(d\) arbitrary, \((\mathsf D_d, \mathsf C_d)\) with \(d\) odd, or \((\mathsf T, \mathsf D_2)\), then \(H\) has a complement in \(K\).
          So, there is a group theoretic section \(\sigma: K/H \to K\) of the canonical projection map.
          This induces a section \(\tilde{\sigma}\) of the map \(\delta_{n - m}\) by embedding \(K\) into \(G_{n - m}(K, K)\) via \(g\mapsto ((g, 1, \dots, 1), \id)\).
          Hence we may choose \(\phi\) in the proposition as the concatenation of the lower part of the diagram.
        \item If \(K\) is either \(\mathsf C_d\) with \(d\) arbitrary or \(\mathsf D_d\) with \(d\) odd and \([K, K] \lneq H \lneq K\), then we can enlarge the diagram in \eqref{eq:compviadiagram} for \([K, K]\) by an additional row \(K/H\xlongrightarrow{\operatorname{inv}}K/H\) with the canonical projections \(K/[K, K]\to K/H\) for the vertical arrows.
          By the previous point, there is a map \(\phi\) that makes the diagram commute for the pair \((K, [K, K])\), hence this map \(\phi\) also works for \(H\).
        \item The remaining cases are the content of Lemma~\ref{lem:imprimspecial}.
      \end{itemize}
      The identification of \(C\) with \(\prod_k G_{b_k}(K, K)\) on \(\Fix_V(P_\lambda)\) follows directly from the construction of the complement in the proof of Proposition~\ref{prop:compviadiagram}.

    \item This is analogous to the proof of \cite[Thm.~3.12~(iii)]{MT18} for \(\alpha = 1\).
      In the case \(\alpha \neq 1\), the first factor \(G_{b_{\lambda_1}}(K, H')\) of \(C\) is embedded into \(G\) by conjugating the first entry by \(\alpha\), where \(\lambda = (\lambda_1,\dots,\lambda_k)\).
      Because conjugation by \(\alpha\) is irrelevant modulo \([K, K]\), this does not change the action of \(C^\circ\).\qedhere
  \end{enumerate}
\end{proof}

\begin{remark}
  The groups \(\mathsf D_d\) with \(d\) even and \(\mathsf O\) are the finite subgroups \(K\) of \(\BBH^\times\) for which \([K, K]\) does not have a complement in \(K\).
  As an a posteriori observation, we notice that exactly for these subgroups \(K\), a parabolic subgroup of \(G_n(K, H)\) does not admit a complement in its normalizer in general.
\end{remark}

\subsection{The exceptional imprimitive groups in rank 2}
\label{sec:imprimexcep}
While the results in this section so far also hold for \(n = 2\), in rank 2 there are further imprimitive groups not of the form \(G_2(K, H)\).
These groups are classified in \cite[Sect.~2]{Coh80} with recent corrections by \cite{Tay25, Wal25}; we summarize the construction.
Let \(K, H\leq \BBH^\times\) with \(H\trianglelefteq K\) be finite groups and let \[\phi :K/H \to K/H\] be a morphism of order 1 or 2.
Then we define the group \[G(K, H, \phi) := \left\{\!\!\begin{pmatrix} x & 0\\ 0 & y\end{pmatrix}, \begin{pmatrix} 0 & x \\ y & 0\end{pmatrix}\ \middle|\ x, y\in K, \phi(xH) = yH\!\right\}\leq\GL_2(\BBH).\]
The group \(G(K, H, \phi)\) acts on \(\BBH^2\) by reflections provided \(K\) is generated by \[L_\phi := \{x\in K \mid \phi(xH) = x^{-1}H\},\] see \cite[Thm.~2.2]{Coh80}.
If \([K, K]\leq H\), it follows from the classification that there is just one such group up to conjugacy, namely \(G_2(K, H) = G(K, H, \operatorname{inv})\).
We now consider the groups \(G(K, H, \phi)\) with \([K, K]\not \leq H\).
These groups are listed in Table~\ref{tab:imprimrank2} where we follow the notation of \cite{Tay25}.
Some of the groups of the series \(G(\mathsf D_m, \mathsf C_l, \psi_r)\) are conjugate, see any of the given references \cite{Coh80, Tay25, Wal25}.
This is irrelevant for what follows, so we do not give further details here.

\begin{table}[t]
  \caption{The imprimitive quaternionic reflection groups \(G(K, H, \phi)\) with \([K, K]\not\leq H\).}
  \label{tab:imprimrank2}
    \def\arraystretch{1.2}
  \begin{tabular}{l|l|l|l}
    \(K\) & \(H\) & \(\phi\) & conditions and remarks\\
    \hline
    \(\mathsf D_m\) & \(\mathsf C_l\) & \(\psi_r:\mathsf D_m / \mathsf C_l \to \mathsf D_m/\mathsf C_l,\) & \(l\mid 2m\), \(1 \leq r \leq n/2\), \(\gcd(n, r) = \gcd(\kappa,\nu) = 1\)\\
    && \(\zeta_{2m}\mathsf C_l\mapsto \zeta_{2m}^r\mathsf C_l,\ \bfj\mathsf C_l\mapsto -\bfj\mathsf C_l\)& where \(n = 2m / l\), \(\kappa = n/\!\gcd(n, r + 1)\)\\
    &&& and \(\nu = n/\!\gcd(n, r - 1)\)\\
    \hline
    \(\mathsf T\) & \(\mathsf C_2\) & \(\rho(\gamma)\) & \(\rho(\gamma)\) denotes conjugation by \(\gamma = \frac{1}{\sqrt2}(\bfi - \bfj)\) \\
    \hline
    \(\mathsf O\) & \(\mathsf D_2\) & \(\id\) & \\
    \(\mathsf O\) & \(\mathsf C_2\) & \(\id\) & \\
    \hline
    \(\mathsf I\) & \(\mathsf C_2\) & \(\id\) & \\
    \(\mathsf I\) & \(\mathsf C_2\) & \(\Theta\) & \(\Theta\in \operatorname{Aut}(\mathsf I)\setminus\operatorname{Inn}(\mathsf I)\), see \cite[Lem.~4.19]{Tay25}\\
    \(\mathsf I\) & 1 & \(\Theta\) &
  \end{tabular}
\end{table}

A parabolic subgroup of \(G(K, H, \phi)\) can still be parametrized by a partial partition \(\lambda\vdash 2\) and an element \(\alpha\in K\) as before for \(G_n(K, H)\).

\begin{lemma}
  \label{lem:imprimrank2}
  Let \(\{1\}\neq P\neq G\) be a parabolic subgroup of \(G = G(K, H, \phi)\).
  Then \(P\) is in one of the following two families.
  \begin{enumerate}
    \item \label{lem:imprimrank2:1}The group \(P\) is conjugate to \(P_{(1)} = H\times\{1\}\).
      In this case, \(P\) has a complement in \(N_G(P)\) if and only if the map \(\phi\) can be lifted to a morphism \(\tilde\phi: K\to K\).
    \item \label{lem:imprimrank2:2}We have \(P = P_{(2)}^\alpha\) for some \(\alpha\in L_\phi\).
      In this case, \(P\) has a complement \(C\) in \(N_G(P)\) and \(C\) is given by \(\{\begin{psmallmatrix} \alpha x\alpha^{-1} & \\ & x\end{psmallmatrix}\mid x\in K, \phi(\alpha x\alpha^{-1} H) = xH\}\).
  \end{enumerate}
\end{lemma}
\begin{proof}
  Every non-trivial parabolic subgroup of \(G\) is of rank 1 and hence consists only of reflections apart from the identity element.
  The reflections in \(G\) are of the form \(\begin{psmallmatrix} h & \\ & 1\end{psmallmatrix}\), \(\begin{psmallmatrix} 1 & \\ & h\end{psmallmatrix}\) with \(h\in H\) or \(\begin{psmallmatrix} & x \\ x^{-1} &\end{psmallmatrix}\) with \(x\in L_\phi\), see \cite[p.~300]{Coh80}.
  The first type of reflections leads to the parabolic subgroups in \ref{lem:imprimrank2:1} and one sees immediately that the groups \(H \times \{1\}\) and \(\{1\}\times H\) are conjugate and form a single conjugacy class.
  A reflection of the second type \(\begin{psmallmatrix} & x \\ x^{-1} &\end{psmallmatrix}\) generates the parabolic subgroup \(P_{(2)}^x\).

  Regarding the existence of a complement, we have for \ref{lem:imprimrank2:1} that the normalizer of \(P_{(1)} = H\times\{1\}\) consists of all matrices \(\begin{psmallmatrix} a & \\ & b\end{psmallmatrix}\in G\).
  The argument is now similar to Proposition~\ref{prop:compviadiagram}.
  Any morphism \(\tilde\phi:K \to K\) lifting \(\phi\) yields a complement \(C = \{\begin{psmallmatrix} \tilde\phi(a) & \\ & a \end{psmallmatrix} \mid a \in K\}\).
  Conversely, any complement \(C\) yields such a morphism by assigning to \(a\in K\) the unique \(b\in K\) such that \(\begin{psmallmatrix} b & \\ & a \end{psmallmatrix}\in C\).

  The normalizer of \(P_{(2)}^x\) coincides with the centralizer of the matrix \(\begin{psmallmatrix} & x \\ x^{-1} &\end{psmallmatrix}\) because \(P_{(2)}^x\) is cyclic of order 2.
  With this the complement given in \ref{lem:imprimrank2:2} follows from a direct computation.
\end{proof}

\begin{defn}
  We say that \(a, b\in K/H\) are \emph{twisted \(\phi\)-conjugate} (or just \emph{\(\phi\)-conjugate}) if there is a \(\xi\in K/H\) with \(\xi^{-1} a \phi(\xi) = b\).
  The equivalence classes of the induced equivalence relation on \(K/H\) are correspondingly called \emph{(twisted) \(\phi\)-conjugacy classes}.
\end{defn}
A short computation shows that the parabolic subgroups \(P_{(2)}^\alpha\) and \(P_{(2)}^\beta\) with \(\alpha,\beta\in K\) are conjugate in \(G\) if and only if \(\alpha H\) and \(\beta H\) are twisted \(\phi\)-conjugate.

In the remainder of this section, we study the parabolic subgroups and their complements for the groups in Table~\ref{tab:imprimrank2} in detail.
We start with the infinite series \(G(\mathsf D_m, \mathsf C_l, \psi_r)\) for which we need the following observation at several places.

\begin{lemma}
  \label{lem:gcdcompute}
  Let \(n, r \in \BBZ_{\geq 0}\) and set \(\kappa = n /\!\gcd(n, r + 1)\) and \(\nu = n /\!\gcd(n, r - 1)\).
  Assume that \(\gcd(\kappa, \nu) = 1\).
  We have \(\epsilon\kappa = \gcd(n, r - 1)\) and \(\epsilon \nu = \gcd(n, r + 1)\), where \[\epsilon = \begin{cases} 1,&\text{if \(\gcd(n, r - 1)\) is odd,}\\ 2,&\text{otherwise.}\end{cases}\]
\end{lemma}
\begin{proof}
  Since \(\gcd(\kappa, \nu) = 1\), we have \[n = \frac{\gcd(n, r + 1)\gcd(n, r - 1)}{\gcd(n, r + 1, r - 1)}.\]
  Then \(\epsilon = \gcd(n, r + 1, r - 1)\) is as claimed.
\end{proof}

\begin{lemma}
  \label{lem:dihedralparasrank2}
  Let \(m, l, r\geq 1\) with the conditions given in Table~\ref{tab:imprimrank2}.
  As in the table, we set \(n = 2m /l\), \(\kappa = n/\!\gcd(n, r + 1)\) and \(\nu = n /\! \gcd(n, r - 1)\).
  The group \(G = G(\mathsf D_m, \mathsf C_l, \psi_r)\) has the following conjugacy classes of non-trivial parabolic subgroups:
  \begin{enumerate}
    \item \label{lem:dihedralparasrank2:1}If \(l > 1\), there is one conjugacy class of \(P_{(1)} = \mathsf C_l \times\{1\}\) of length 2.
    \item \label{lem:dihedralparasrank2:2}There are one or two conjugacy classes of groups \(P_{(2)}^\alpha\) with \(\alpha = \zeta_{2m}^k\):
      \begin{enumerate}[(a)]
        \item If \(\frac{2m}{\gcd(n, r - 1)}\) is odd or \(\gcd(n, r - 1)\) is odd, there is one such class of length \(l\gcd(n, r + 1)\) represented by \(P_{(2)}^1\).
        \item Otherwise, there are two classes, each of length \(\frac{1}{2}l\gcd(n, r + 1)\), represented by \(P_{(2)}^1\) and \(P_{(2)}^{\zeta_{2m}^\kappa}\), respectively.
      \end{enumerate}
    \item \label{lem:dihedralparasrank2:3}Likewise, there are one or two conjugacy classes of groups \(P_{(2)}^\alpha\) with \(\alpha = \zeta_{2m}^k\bfj\):
      \begin{enumerate}[(a)]
        \item If \(\frac{2m}{\gcd(n, r + 1)}\) is odd or \(\gcd(n, r + 1)\) is odd, there is one such class of length \(l\gcd(n, r - 1)\) represented by \(P_{(2)}^\bfj\).
        \item Otherwise, there are two classes, each of length \(\frac{1}{2}l\gcd(n, r - 1)\), represented by \(P_{(2)}^\bfj\) and \(P_{(2)}^{\zeta_{2m}^\nu \bfj}\), respectively.
      \end{enumerate}
  \end{enumerate}
\end{lemma}
\begin{proof}
  The separation of the conjugacy classes in the types \(P_{(1)}\) and \(P_{(2)}^\alpha\) is Lemma~\ref{lem:imprimrank2}.

  We are left with studying the conjugacy between the groups \(P_{(2)}^\alpha\) in \ref{lem:dihedralparasrank2:2} and \ref{lem:dihedralparasrank2:3}.
  By \cite[Lem.~4.9,~4.10]{Tay25}, the set \(L_{\psi_r}\) can be partitioned as \[L_{\psi_r} = L_1 \mathbin{\dot\cup} L_2 := \langle \zeta_{2m}^\kappa\rangle \mathbin{\dot\cup} \langle \zeta_{2m}^\nu\rangle \bfj.\]
  Hence there are \(2m/\kappa = l\gcd(n, r + 1)\) parabolic subgroups of the form \(P_{(1)}^\alpha = \left\langle\begin{psmallmatrix} & \alpha\\\alpha^{-1}& \end{psmallmatrix}\right\rangle\) with \(\alpha\in L_1\) and \(2m/\nu = l\gcd(n, r - 1)\) groups of this form with \(\alpha\in L_2\).
    Two groups \(P_{(1)}^\alpha\) and \(P_{(1)}^\beta\) with \(\alpha,\beta\in L_{\psi_r}\) are conjugate in \(G\) if and only if the cosets \(\alpha\mathsf C_l\) and \(\beta\mathsf C_l\) in \(\mathsf D_m/\mathsf C_l\) are twisted \(\psi_r\)-conjugate in \(\mathsf D_m / \mathsf C_l\).
  One sees directly that elements in \(L_1\) cannot be twisted \(\psi_r\)-conjugate to elements in \(L_2\) and this gives the separation of parabolic subgroups into \ref{lem:dihedralparasrank2:2} and \ref{lem:dihedralparasrank2:3}.

  We proceed with proving the further statements in \ref{lem:dihedralparasrank2:2}.
  An element \(g\in\mathsf D_m\) is \(\psi\)-conjugate to \(\alpha = 1\) modulo \(\mathsf C_l\) if \(g\mathsf C_l = \xi^{-1}\psi_r(\xi)\) for some \(\xi\in \mathsf D_m/\mathsf C_l\).
  Hence the \(\psi\)-conjugacy class of \(1\) modulo \(\mathsf C_l\) is the set \[[1]_{\psi_r} = \{\zeta_{2m}^{(r - 1)a + nb}, \zeta_{2m}^{(r - 1)a + m + nb}\mid 0\leq a < 2m, 0 \leq b < l\}.\]
  We may abbreviate this to \([1]_{\psi_r} = \langle \zeta_s\rangle \cup (-1)\langle \zeta_s\rangle\) with \(s = \operatorname{lcm}(2m / \gcd(2m, r - 1), l)\).
  Notice that this union is only disjoint if \(s\) is odd and \([1]_{\psi_r} = \langle \zeta_s\rangle\) if \(s\) is even, so we have \[|{[1]_{\psi_r}}| = \begin{cases} s, & \text{ if \(s\) even},\\ 2s, & \text{ if \(s\) odd.}\end{cases}\]
  One checks that \(s = 2m /\!\gcd(n, r - 1) = \nu l\).
  Comparing \(|{[1]_{\psi_r}}|\) with the cardinality of \(L_1 = \langle\zeta_{2m}^\kappa\rangle\), we obtain \[l\gcd(n, r + 1) = \epsilon s\] with \(\epsilon\) as in Lemma~\ref{lem:gcdcompute}.
  Because \(2m\) is even, \(s\) and \(\gcd(n, r - 1)\) cannot both be odd.
  So, if either \(s\) is odd or \(\gcd(n, r - 1)\) is odd, we have \(|{[1]_{\psi_r}}| = l\gcd(n, r + 1)\) and there is just one conjugacy class in \ref{lem:dihedralparasrank2:2}.
  Finally, if both \(s\) and \(\gcd(n, r - 1)\) are even, the set \([1]_{\psi_r} = \langle \zeta_{2m}^{2\kappa}\rangle\) only constitutes for half of the elements in \(L_1\).
  Notice that in this case \(|{L_1}| = l\gcd(n, r + 1)\) is indeed even as \(\gcd(n, r - 1)\) is even if and only if \(\gcd(n, r + 1)\) is even.
  We have \([\zeta_{2m}^\kappa]_{\psi_r} = \zeta_{2m}^\kappa\langle \zeta_s\rangle \cup (-\zeta_{2m}^{-\kappa})\langle \zeta_s\rangle\).
  Because \(\zeta_{2m}^{2\kappa}\in \langle \zeta_s\rangle\), we obtain \(|{[\zeta_{2m}^\kappa]_{\psi_r}}| = |{[1]_{\psi_r}}|\), which finishes the proof of \ref{lem:dihedralparasrank2:2}.
  The claim in \ref{lem:dihedralparasrank2:3} follows analogously.
\end{proof}

In the following proposition, we identify the complements with their isomorphism type as a quaternionic reflection group by slight abuse of notation.
To obtain the concrete subgroups of \(G\) one may use the explicit description of the complements in Lemma~\ref{lem:imprimrank2}.
\begin{prop}
  \label{prop:dihedralparasrank2}
  Keep the notation of Lemma~\ref{lem:dihedralparasrank2}.
  Every parabolic subgroup of \(G = G(\mathsf D_m, \mathsf C_l, \psi_r)\) has a complement in its normalizer in \(G\).
  Specifically, we have:
  \begin{enumerate}
    \item \label{prop:dihedralparasrank2:1}\(N_G(P_{(1)}) = P_{(1)} \rtimes \mathsf D_m\).
    \item \label{prop:dihedralparasrank2:2}If \(\frac{2m}{\gcd(n, r - 1)}\) is odd, we have \(N_G(P_{(2)}^1) = P_{(2)}^1\rtimes \mathsf C_{l\gcd(n, r - 1)}\).
      Otherwise, we have \(N_G(P_{(2)}^\alpha) = P_{(2)}^\alpha\rtimes \mathsf D_{\frac{1}{2}l\gcd(n, r - 1)}\) with \(\alpha \in \{1,\zeta_{2m}^\kappa\}\).
    \item \label{prop:dihedralparasrank2:3}If \(\frac{2m}{\gcd(n, r + 1)}\) is odd, we have \(N_G(P_{(2)}^\bfj) = P_{(2)}^\bfj\rtimes \mathsf C_{l\gcd(n, r + 1)}\).
      Otherwise, we have \(N_G(P_{(2)}^\alpha) = P_{(2)}^\alpha\rtimes \mathsf D_{\frac{1}{2}l\gcd(n, r + 1)}\) with \(\alpha \in \{\bfj,\zeta_{2m}^\nu \bfj\}\).
  \end{enumerate}
\end{prop}
\begin{proof}
  \begin{enumerate}
    \item By Lemma~\ref{lem:imprimrank2} \ref{lem:imprimrank2:1}, we only need to show that there is a morphism \(\tilde \phi:\mathsf D_m \to \mathsf D_m\) that lifts \(\psi_r\).
      If \(r\) is odd, the map \(\psi_r\) can immediately be lifted by \(\zeta_{2m} \mapsto \zeta_{2m}^r\) and \(\bfj \mapsto -\bfj\).
      If \(r\) is even, then \(n\) is odd because \(\gcd(n, r) = 1\).
      So, in this case, we may choose \(\zeta_{2m} \mapsto \zeta_{2m}^{r + n}\) and \(\bfj\mapsto -\bfj\).

    \item We first consider the case where there is just one conjugacy class represented by \(P = P_{(2)}^1\) and \(|{P^G}| = l\gcd(n, r + 1)\).
      We compute \[|{N_G(P)/P}| = \frac{|{G}|}{|{P}||{P^G}|} = \frac{4ml}{l\gcd(n, r + 1)} = 2l\kappa\] for the order of a complement of \(P\) in the normalizer.

      By Lemma~\ref{lem:imprimrank2} \ref{lem:imprimrank2:2}, a complement \(C\) of \(P\) is given by the group consisting of the matrices \(\begin{psmallmatrix} x & \\ & x\end{psmallmatrix}\in G\).
      To identify this group \(C\), we need to find all \(x\in \mathsf D_m\) with \(\psi_r(x\mathsf C_l) = x\mathsf C_l\).
      Let \(a\in \{0,\dots, 2m - 1\}\) such that \(\psi_r(\zeta_{2m}^a\mathsf C_l) = \zeta_{2m}^a\mathsf C_l\).
      Then there is a \(b\in \{0,\dots, l - 1\}\) such that \(2m\mid (a(r - 1) - b n)\).
      From \(n \mid 2m\), we conclude \(n\mid a(r - 1)\), so \(\nu\mid a\).
      Hence the group \(\mathsf C_{l\gcd(r, n - 1)} = \langle \zeta_{2m}^\nu\rangle\) can be embedded into \(C\).
      If \(\frac{2m}{\gcd(n, r - 1)}\) is odd, then \(\gcd(n, r - 1)\) must be even.
      Hence, in this case, we have \[|{C}| = 2l\kappa = l\gcd(r, n - 1)\] by Lemma~\ref{lem:gcdcompute}.

      Otherwise, \(\frac{2m}{\gcd(n, r - 1)}\) is even, so \(\gcd(n, r - 1)\mid m\) and there are \(u, v\in \BBZ\) with \(un + v(r - 1) = m\).
      Then we have \[\psi_r(\zeta_{2m}^{-u}\bfj\mathsf C_l) = \zeta_{2m}^{-u - u(r - 1) - vn + m}\bfj\mathsf C_l = \zeta_{2m}^{-u}\bfj\mathsf C_l,\] so \(\zeta_{2m}^{-u}\bfj\in C\) and \(C \cong \mathsf D_{\frac{1}{2}l\gcd(r, n - 1)}\).

      If there are two conjugacy classes with cardinality \(|{P^G}| = \frac{1}{2}l\gcd(n, r - 1)\), then \[|{N_G(P)/P}| = 4l\kappa = 2 l \gcd(r, n - 1)\] by Lemma~\ref{lem:gcdcompute}.
      As above we obtain \(C \cong \mathsf D_{\frac{1}{2}l\gcd(r, n - 1)}\) for the complement.

    \item This is analogous to \ref{prop:dihedralparasrank2:2}.\qedhere
  \end{enumerate}
\end{proof}

\begin{table}[h]
  \caption{Conjugacy classes of the parabolic subgroups \(P_{(2)}^\alpha\) in \(G(K, H, \phi)\) and complements \(C\) with \(N_G(P_{(2)}^\alpha) = P_{(2)}^\alpha \rtimes C\).}
  \label{tab:parasimprimrank2}
    \def\arraystretch{1.2}
  \begin{tabular}{l|l||l|l|l}
    \(G\) & number of classes & \(\alpha\) & length of conjugacy class & \(C\) \\
    \hline
    \(G(\mathsf T, \mathsf C_2, \rho(\gamma))\) & 1 & 1 & 12 & \(\mathsf C_4\)\\
    \hline
    \(G(\mathsf O, \mathsf D_2, \id)\) & 2 & 1 & 8 & \(\mathsf O\)\\
    &                                     & \(\zeta_8\) & 24 & \(\mathsf D_4\)\\
    \hline
    \(G(\mathsf O, \mathsf C_2, \id)\) & 3 & 1 & 2 & \(\mathsf O\)\\
                                       &   & \(\bfi\) & 6 & \(\mathsf D_4\)\\
                                       &   & \(\zeta_8 \bfj\) & 12 & \(\mathsf D_2\)\\
    \hline
    \(G(\mathsf I, \mathsf C_2, \id)\) & 2 & 1 & 2 & \(\mathsf I\)\\
                                       &   & \(\bfj\) & 30 & \(\mathsf D_2\)\\
    \hline
    \(G(\mathsf I, \mathsf C_2, \Theta)\) & 1 & 1 & 20 & \(\mathsf D_3\)\\
    \hline
    \(G(\mathsf I, 1, \Theta)\) & 1 & 1 & 20 & \(\mathsf C_6\)
  \end{tabular}
  \vskip 1ex
  {\raggedright Notation: \(\rho(\gamma)\) denotes conjugation by \(\gamma\) and \(\Theta\) is the morphism from \cite[Lem.~4.19]{Tay25}.\par}
\end{table}

Finally, we have the following result regarding the remaining groups in Table~\ref{tab:imprimrank2}.
As in Proposition~\ref{prop:dihedralparasrank2}, the complements are described via their isomorphism type as reflection groups by slight abuse of notation.
\begin{prop}
  \label{prop:imprimrank2excep}
  Let \(G = G(K, H, \phi)\) be one of the groups in Table~\ref{tab:imprimrank2} with \(K\in\{\mathsf T, \mathsf O, \mathsf I\}\).
  \begin{enumerate}
    \item \label{prop:imprimrank2excep:1}The parabolic subgroup \(P = P_{(1)}\leq G\) has a complement in its normalizer and we have \(N_G(P) = P\rtimes K\).
    \item \label{prop:imprimrank2excep:2}The number of conjugacy classes of parabolic subgroups of the groups \(P_{(2)}^\alpha\) and their complements are listed in Table~\ref{tab:parasimprimrank2}.
  \end{enumerate}
\end{prop}
\begin{proof}
  For \ref{prop:imprimrank2excep:1}, by Lemma~\ref{lem:imprimrank2}, it suffices to check that \(\phi\) lifts to a morphism \(K\to K\).
  This is clearly the case, as the maps in question are already defined as morphisms of \(K\) in Table~\ref{tab:imprimrank2}.

  To determine the conjugacy classes of the groups \(P_{(2)}^\alpha\) in \ref{prop:imprimrank2excep:2}, one needs to determine the partition of \(L_\phi\) into twisted \(\phi\)-conjugacy classes as in Lemma~\ref{lem:dihedralparasrank2}.
  To determine the complement, one then uses Lemma~\ref{lem:imprimrank2} \ref{lem:imprimrank2:2}.
  This and the remaining results in Table~\ref{tab:parasimprimrank2} follow from direct computation; we give some details.

  We denote by \(\rho(\gamma)\) conjugation by \(\gamma = \frac{1}{\sqrt 2}(\bfi - \bfj)\).
  For \(G = G(\mathsf T, \mathsf C_2, \rho(\gamma))\), we have \(\mathsf T/\mathsf C_2\cong A_4\), the alternating group on four letters.
  Further, \(\rho(\gamma)\) corresponds to \(\rho((1, 2))\), that is, conjugation by the transposition \((1, 2)\), under this isomorphism, see \cite[Table~I]{Coh80}.
  There are six elements \(\sigma\in A_4\) with \(\sigma^{(1, 2)} = \sigma^{-1}\), namely \(\id\), \((1,2)(3, 4)\) and the four 3-cycles moving both 1 and 2.
  Hence \(|{L_{\rho(\gamma)}}| = 12\).
  If \(\sigma\) is one of the mentioned 3-cycles, then \(\sigma^{-1} \sigma^{(1, 2)} = \sigma\), so \(\sigma \in [\id]_{\rho((1, 2))}\), the twisted conjugacy class of \(\id\).
  Likewise, we have \(x \rho((1, 2))(x)^{-1} = (1, 2)(3, 4)\in [\id]_{\rho((1, 2))}\) with \(x = (1,4)(2,3)\) and so there is only one twisted \(\rho(\gamma)\)-conjugacy class in \(L_{\rho(\gamma)}\).
  The complement \(C\) of \(P_{(2)}^1\) must be of order 4, hence \(C\) acts as \(\mathsf C_4\) on the fixed space of \(P\).

  In case \(G = G(\mathsf O, \mathsf D_2, \id)\), we have \(\mathsf O/\mathsf D_2 = S_3\).
  The elements of \(L_\id\) correspond to elements of order at most 2 in \(S_3\) of which there are four and so \(|{L_\id}| = 32\).
  Because the defining morphism of \(G\) is the identity, twisted conjugacy is just regular conjugacy.
  Then there are two conjugacy classes of parabolic subgroups: one corresponding to the identity in \(S_3\) and one corresponding to the transpositions in \(S_3\).
  For \(P_{(2)}^1\), the complement in the normalizer must then be isomorphic to \(\mathsf O\).
  The other conjugacy class of parabolic subgroups is for example represented by \(P_{(2)}^{\zeta_8}\).
  The complement of \(P_{(2)}^{\zeta_8}\) corresponds to all elements of \(S_3\) that are stable under conjugacy by \(\zeta_8\mathsf D_2\), which acts as a transposition.
  Hence there are two such elements in \(S_3\) and they give the group \(\langle \mathsf D_2, \zeta_8\rangle = \mathsf D_4 \leq \mathsf O\).

  Analogously to the previous case, we have \(\mathsf O/\mathsf C_2 = S_4\) for \(G = G(\mathsf O,\mathsf C_2, \id)\).
  We get \(|{L_\id}| = 20\) and three conjugacy classes,\footnote{Notice that there is a misprint in \cite[Table~I]{Coh80}, which lists \(|{L_\id}| = 14\).} which correspond to the identity element, the transpositions and the products of two transpositions, respectively.
  The identification of the complements works as above.

  For \(G = G(\mathsf I,\mathsf C_2, \id)\), we have \(\mathsf I/\mathsf C_2 = A_5\) and elements of \(L_\id\) correspond to elements of order at most 2 in \(A_5\).
  So, we have \(|{L_\id}| = 32\) and there are two conjugacy classes, one for the identity and one containing all products of two transpositions.
  The complements can now be identified analogously to the above.

  For \(G = G(\mathsf I, \mathsf C_2, \Theta)\), we again work in \(\mathsf I/\mathsf C_2 = A_5\).
  The morphism \(\Theta\) corresponds to conjugation by \((1, 2)\) in \(A_5\), see \cite[Table~I]{Coh80}.
  We have \(|{L_\Theta}| = 20\) because there are ten elements \(\sigma\in A_5\) with \(\sigma^{(1, 2)} = \sigma^{-1}\), namely the identity, the three products of \((1, 2)\) with another transposition and six 3-cycles that move both 1 and 2.
  One checks in the same way as for the group \(G(\mathsf T, \mathsf C_2, \rho(\gamma))\) above that all of these are in a single twisted \(\Theta\)-conjugacy class.

  Finally, for \(G = G(\mathsf I, 1, \Theta)\), one checks that \(|{L_\Theta}| = 20\) and there is just one conjugacy class as before.
  The complement of \(P_{(2)}^1\) must then have order 6, so it must act as \(\mathsf C_6\) on the fixed space.
\end{proof}

\section{Primitive groups with imprimitive complexification}
\label{sec:primimprim}

We now consider primitive quaternionic reflection groups with imprimitive complexification.
These groups are classified in \cite[Sect.~3]{Coh80}; however, by \cite[Thm.~7.9]{Tay25} three groups in Cohen's list are actually imprimitive.

\begin{notation}
  For \(d\in\BBZ_{\geq 1}\), let \(\mu_d \leq \BBC^\times\) be the group of all \(d\)-th roots of unity.
  We often identify \(\mu_d\) and \(\mu_d\cdot I_2\leq \GL_2(\BBC)\) by abuse of notation.
\end{notation}

Throughout, let \[s := \begin{pmatrix} 0 & -\bfj \\ \bfj & 0\end{pmatrix}\in\GL_2(\BBH).\]
Let \(G\) be a primitive, irreducible quaternionic reflection group with imprimitive complexification of rank larger than 1.
By \cite[Thm.~3.6]{Coh80}, there is a group \(H\leq \GU_2(\BBC)\) such that \[G = E(H)\text{ where }E(H) := \langle H, s\rangle.\]
In particular, the group \(G\) is of rank 2.

The group \(H\) in question has more structure.
To be able to fix further notation, we make the following observation.
\begin{lemma}
  \label{lem:centre}
  Let \(H\leq \GL_2(\BBC)\) be a finite group that acts primitively on \(\BBC^2\).
  Then up to conjugacy there is a \(d\geq 1\) with \(Z(H) = \mu_d \cdot I_2\).
\end{lemma}
\begin{proof}
  The matrices in the centre \(Z(H)\) commute, so they are simultaneously diagonalizable.
  By conjugating \(H\) appropriately, we may hence assume that any \(g\in Z(H)\) is of the form \(g = \begin{psmallmatrix} \zeta & \\ & \zeta'\end{psmallmatrix}\) for some roots of unity \(\zeta, \zeta'\in \BBC^\times\).
  Because \(H\) acts primitively, there must be an element \(\begin{psmallmatrix}a & b\\ c & d\end{psmallmatrix}\in H\) with \(b\neq 0\) or \(c \neq 0\).
  Then \[\begin{pmatrix} \zeta a & \zeta b \\ \zeta' c & \zeta' d\end{pmatrix} = \begin{pmatrix}\zeta a & \zeta' b\\ \zeta c& \zeta'd\end{pmatrix},\] because \(g\in Z(H)\), so \(\zeta = \zeta'\).

  Let \(\zeta_k, \zeta_l\in \BBC^\times\) be a primitive \(k\)-th, respectively, \(l\)-th root of unity.
  If \(\zeta_kI_2, \zeta_lI_2\in Z(H)\), then \(\zeta_qI_2\in Z(H)\) with \(p = \operatorname{lcm}(k, l)\) and \(q = p/\!\gcd(p/k, p/l)\).
  So \(\zeta_k, \zeta_l\in\langle \zeta_q\rangle\) and we see that \(Z(H)\) must be cyclic.
\end{proof}

\begin{notation}
  \label{not:primimprim}
  Let \(G = \langle H, s\rangle\) be a primitive group with imprimitive complexification.
  By \cite[Lem.~3.3]{Coh80}, we have \(H = \mu_{d'} \cdot H_0\) where \[H_0\leq\GU_2(\BBC)\] is a primitive complex reflection group and \(d'\geq 1\) such that \((\det h)^{d'} = 1\) for every \(h\in H_0\).
  In the following, we replace \(\mu_{d'}\) by the centre of \(H\) and write \(H = \mu_d \cdot H_0\) with \(d\geq 1\) such that \(Z(H) = \mu_d I_2\) as in Lemma~\ref{lem:centre}.
  By \cite[Lem.~3.3]{Coh80}, the possible groups \(H_0\) are \[G_5,G_7,G_8,\dots, G_{22}\] from \cite{ST54}.
  See \cite[Lem.~3.3]{Coh80} and \cite[Thm.~7.9]{Tay25} for the possible values of \(d\) depending on \(H_0\).
\end{notation}

In the following, we fix a group \(G\) and the corresponding groups \(H\), \(H_0\) and \(\mu_d\) as in Notation~\ref{not:primimprim}.

\subsection{The reflections}
We identify all reflections in \(G\).

\begin{lemma}
  \label{lem:refltype}
  Let \(g\in G\) be a reflection.
  Then \(g\) belongs of one of the following two types.
  \begin{enumerate}[(a)]
    \item\label{refltype:a} The element \(g \in H_0\) is a complex reflection.
    \item\label{refltype:b} We have \(g = (\zeta I_2) s\) with \(\zeta\in\mu_d\).
  \end{enumerate}
\end{lemma}
\begin{proof}
  The elements of \(G\) are of the form \(h\) and \(hs\) with \(h\in H\).
  If \(h\in H\) is a quaternionic reflection, then \(h\) must be a complex reflection, so \(h\in H_0\).

  Let now \(h\in H\) be such that \(hs\in G\) is a quaternionic reflection.
  Let \(0\neq v = (v_1, v_2)^\top \in\Fix(hs)\), so \(hs v = v\).
  Write \(v_i = u_i + w_i\bfj\) with \(u_i, w_i\in \BBC\).
  Using \(h\in \GU_2(\BBC)\), one computes
  \begin{align*}
    \frac{1}{\det(h)}h^2 \begin{pmatrix}u_1\\ u_2\end{pmatrix} &= h \begin{pmatrix} 0 & 1 \\ -1 & 0\end{pmatrix} \overline{h} \begin{pmatrix} 0 & -1 \\ 1 & 0\end{pmatrix} \begin{pmatrix}u_1\\u_2\end{pmatrix} = \begin{pmatrix}u_1\\u_2\end{pmatrix},\\
    \frac{1}{\det(h)}h^2\begin{pmatrix}w_1\\w_2\end{pmatrix} &= \begin{pmatrix}w_1\\w_2\end{pmatrix},
  \end{align*}
  where \(\overline h\) denotes the complex conjugate of \(h\).
  By \cite[Lem.~3.3]{Coh80}, we can write \(h = (\zeta I_2) \tilde h\) where \(\zeta\in \BBC\) is a root of unity and \(\tilde h\in \SL_2(\BBC)\).
  Then we have \(\frac{1}{\det(h)}h^2 = \tilde h^2\) and \(\tilde h^2\) has an eigenvector with eigenvalue 1 because \((u_1,u_2)^\top\) and \((w_1,w_2)^\top\) cannot both be zero.
  Since \(\tilde h\) is of finite order, this implies \(\tilde h^2 = I_2\).
  Hence \(\tilde h = \pm I_2\) and \(hs = \pm(\zeta I_2) s\) as required.
\end{proof}

In the following, we refer to the two families of reflections in Lemma~\ref{lem:refltype} as reflections of type \ref{refltype:a} and \ref{refltype:b}, respectively.

\begin{lemma}
  \label{lem:sconj}
  For \(g\in \GU_2(\BBC)\), we have \(sgs = \overline{\det(g)} g\).
\end{lemma}
\begin{proof}
  Write \(g = \begin{psmallmatrix} a &b \\ c & d\end{psmallmatrix}\).
  Computing the product gives \[sgs = \begin{pmatrix} \overline{d} & -\overline{c} \\ -\overline{b} & \overline{a}\end{pmatrix},\] because \(\bfj x\bfj = -\overline{x}\) for all \(x\in \BBC\).
  Because \(g\in\GU_2(\BBC)\), we have \[\overline{g} = (g^{-1})^\top = \det(g)^{-1}\begin{pmatrix} d & -c \\ -b & a\end{pmatrix}.\]
  Hence \(sgs = \overline{\det(g)} g\) as claimed.
\end{proof}

Next we determine the conjugacy classes of reflections in \(G\).
\begin{lemma}
  \label{lem:reflconj}
  \begin{enumerate}
    \item Reflections of type~\ref{refltype:a} are not conjugate to reflections of type~\ref{refltype:b}.
    \item Let \(g_1,\dots, g_t\in G\) be representatives of the \(G\)-conjugacy classes of reflections in \(G\) of type~\ref{refltype:a}.
      Assume that \(g_1,\dots, g_r\) for some \(r\in\{0,\dots, t\}\) are of order 2 and \(g_{r + 1},\dots, g_t\) are of higher order.
      Then the elements \(g_1,\dots, g_t\in H_0\) together with \(sg_{r + 1}s,\dots, sg_t s\in H_0\) are representatives of the \(H_0\)-conjugacy classes of reflections in \(H_0\).
    \item For \(H_0\in \{G_8, G_{10}, G_{12}, G_{14}\}\), there is one conjugacy class of reflections of type~\ref{refltype:b} with representative \(s\).
      For all other \(H_0\), there are two conjugacy classes of reflections of type~\ref{refltype:b} with representatives \(s\) and \((\zeta_dI_2) s\) where \(\zeta_d\) is a primitive \(d\)-th root of unity with \(d = |{Z(H)}|\).
   \end{enumerate}
\end{lemma}
\begin{proof}
  \begin{enumerate}
    \item Conjugacy in \(G\) leaves the block structure of the matrices unchanged, hence a reflection of type~\ref{refltype:a} cannot be conjugate to a reflection of type~\ref{refltype:b}.

    \item Let \(h_1,h_2\in H_0\) be complex reflections.
      Because \(\mu_d = Z(H)\), the matrices \(h_1\) and \(h_2\) are conjugate in \(H_0\) if and only if they are conjugate in \(H\).

      If \(h_1\) and \(h_2\) are conjugate in \(H_0\), then they are conjugate in \(G\) as well.
      Hence we must find representatives of the conjugacy classes of reflections of type~\ref{refltype:a} in \(G\) among the representatives of the conjugacy classes of complex reflections in \(H_0\).

      Let \(h_1\) and \(h_2\) be conjugate in \(G\), so there is a \(g\in G\) with \(gh_1 g^{-1} = h_2\).
      If \(g\in H\), then \(h_1\) and \(h_2\) are conjugate in \(H\), so in \(H_0\).
      Assume that \(g = hs\) for some \(h\in H\).
      With Lemma~\ref{lem:sconj}, we compute \[gh_1g^{-1} = hsh_1sh^{-1} = h \overline{\det(h_1)}h_1 h^{-1},\] so \(h_2\) is conjugate in \(H\) to \(\overline{\det(h_1)} h_1 = sh_1s\).

      It remains to understand, when \(h_1\) is conjugate to \(\overline{\det(h_1)}h_1\) in \(H\).
      The matrix \(h_1\) is a reflection, so \(\det(h_1)\) is a root of unity and we have \(\overline{\det(h_1)} = \det(h_1)^{-1}\).
      In particular, \(\det(\det(h_1)^{-1}h_1) = \det(h_1)^{-1}\), so \(h_1\) and \(\overline{\det(h_1)}h_1\) can only be conjugate in \(H\), if \(\det(h_1) = \det(h_1)^{-1}\), that is, if \(h_1\) is of order 2 and \(\det(h_1) = -1\).
      One checks with \cite[Table~2]{OS82} that in the possible groups \(H_0\) every reflection \(r\) of order 2 is conjugate to \(-r\).
      Hence \(h_1\) is conjugate to \(sh_1s\) in \(H\) if and only if \(h_1\) is of order 2.

    \item Let \(r = (\zeta I_2)s\) be a reflection of type~\ref{refltype:b}.
      Then \(\zeta = \zeta_d^k\) for some \(0\leq k\leq d - 1\).
      Let \(l = \lceil (d - k)/2\rceil\).
      Then \((\zeta_d^l I_2) (\zeta I_2) s (\zeta_d^{-l} I_2)= (\zeta_d^l\zeta\zeta_d^l I_2)s \in \{s, (\zeta_dI_2) s\}\), so every reflection of type~\ref{refltype:b} is conjugate to one of the given representatives.

      On the other hand, assume that there is a \(g\in G\) with \(gsg^{-1} = (\zeta_d I_2)s\).
      We have \(g = h\) or \(g = h s\) for some \(h\in H\) and in either case it follows that \(gsg^{-1} = hsh^{-1}\).
      Then \(hsh^{-1} = h\overline{\det(h^{-1})}h^{-1} s\) by Lemma~\ref{lem:sconj}, so \(\det(h) = \overline{\det(h^{-1})} = \zeta_d\).
      One checks with the descriptions of the groups given in \cite[(3.6)]{Coh76} that such an element \(h\) exists exactly in the listed cases, namely \(H_0 \in \{G_8, G_{10}, G_{12}, G_{14}\}\), so that \(s\) and \((\zeta_d I_2)s\) are conjugate for these groups \(H_0\).
      In the other cases, \(s\) and \((\zeta_d I_2)s\) are not conjugate in \(G\).\qedhere
  \end{enumerate}
\end{proof}

\subsection{The parabolic subgroups}

We classify the parabolic subgroups in \(G\).

\begin{lemma}
  \label{lem:paratype}
  Let \(P\leq G\) be a parabolic subgroup and assume \(\{1\}\neq P\neq G\).
  Then \(P\) is in one of the following two families:
  \begin{enumerate}[(a)]
    \item\label{paratype:a} We have \(P\leq H_0\) and \(P\) is a (complex) parabolic subgroup of \(H_0\).
    \item\label{paratype:b} There is a root of unity \(\zeta\in \mu_d\) with \(P = \langle (\zeta I_2) s\rangle\) and \(P\) is cyclic of order 2.
  \end{enumerate}
\end{lemma}
\begin{proof}
  The group \(P\) acts on \(\BBH^2\) and fixes a quaternionic subspace of dimension 1 pointwise, hence every \(g\in P\setminus\{1\}\) is a quaternionic reflection.
  Assume that \(P\) contains both an element \(1\neq h\in H\) and an element \((\zeta I_2) s\).
  Then \((\zeta I_2)h s\in P\) and this must be a reflection, so \((\zeta I_2)h = \zeta' I_2\) for some root of unity \(\zeta'\in \BBC\).
  Hence \(h = \zeta''I_2\) for a root of unity \(\zeta''\in \BBC\), so \(h\) is not a reflection, a contradiction.
  We conclude that the reflections in \(P\) are either all of type~\ref{refltype:a} or all of type~\ref{refltype:b}.

  If \(P\) contains only reflections of type~\ref{refltype:a}, then \(P\leq H_0\).
  Further, \(P\) must fix a subspace of dimension 1 in \(\BBC^2\), so \(P\) is a parabolic subgroup of \(H_0\).

  Finally, assume that \(P\) contains only reflections of type~\ref{refltype:b} and let \((\zeta I_2) s, (\zeta'I_2) s\in P\).
  One computes \((\zeta I_2) s (\zeta' I_2)s = (\zeta (\zeta')^{-1}I_2)s^2 = \zeta (\zeta')^{-1}I_2\).
  Again, every non-trivial element in \(P\) must be a reflection, so \(\zeta(\zeta')^{-1} = 1\).
  We conclude \(\zeta = \zeta'\), so \(P = \langle I_2, (\zeta I_2)s\rangle\) is cyclic of order 2.
\end{proof}

We again refer to the two families of parabolic subgroups in Lemma~\ref{lem:paratype} as parabolic subgroups of type \ref{paratype:a} and \ref{paratype:b}, respectively.
Because conjugation in \(G\) maintains the block structure of the matrices, a subgroup of type~\ref{paratype:a} and a subgroup of type~\ref{paratype:b} cannot be conjugate.
By Lemma~\ref{lem:reflconj}, there are one or two conjugacy classes of parabolic subgroups of type~\ref{paratype:b} in \(G\).

\begin{prop}
  \label{prop:paraconj}
  The conjugacy classes of parabolic subgroups of the primitive, irreducible quaternionic reflection groups with imprimitive complexification are listed in Table~\ref{tab:paras}.
\end{prop}
\begin{proof}
  We first show that it suffices to consider the groups \(E(H)\) with \(H = H_0\) a complex reflection group.
  Let \(G = E(H)\), \(G' = E(H')\) for some \(H, H'\leq \GU_2(\BBC)\) and assume that the subgroups generated by complex reflections \(H_0\leq H\) and \(H_0'\leq H'\) coincide.
  Then the conjugacy classes of parabolic subgroups of type~\ref{paratype:a} of \(G\) and \(G'\) coincide as well.
  Indeed, the parabolic subgroups of type~\ref{paratype:a} depend only on \(H_0 = H_0'\) by Lemma~\ref{lem:paratype} and their conjugacy is independent of \(Z(H)\) or \(Z(H')\).

  Hence Table~\ref{tab:paras} only needs to list the groups \(E(H_0)\) for the finitely many groups \(H_0\).
  The parabolic subgroups of these are classified in \cite[Table~2]{OS82} (see also \cite[Tables C1 and C2]{OT92}).
  Together with Lemma~\ref{lem:reflconj}, we can infer the conjugacy classes of parabolic subgroups for the groups \(E(H_0)\).
\end{proof}

\begin{remark}
  \label{rem:tabexplain}
  We explain the notation of Table~\ref{tab:paras}.
  A row of Table~\ref{tab:paras} corresponds to the group \(G = E(H)\) with \(H = \mu_d H_0\).
  The first three columns list the possible groups \(H_0\) in the notation from \cite{ST54}, their orders and the minimal possible value for \(d\).
  Usually the later is the order of \(Z(H_0)\) except for the three cases excluded by \cite[Thm.~7.9]{Tay25}.

  The remaining columns list the length of a conjugacy class of parabolic subgroups of the given isomorphism type.
  The first five isomorphism types are parabolic subgroups of type~\ref{paratype:a}.
  In some cases there are distinct conjugacy classes of groups isomorphic to \(C_2\); we distinguish these via the labels \(C_2'\) and \(C_2''\) as in \cite{OS82}.
  This part of the table is identical to \cite[Table~2]{OS82} with the exception that two conjugacy classes of type \(C_3\) in \(G_5\) respectively \(G_7\) are fused in \(E(G_5)\) and \(E(G_7)\).
  The remaining two columns list conjugacy classes of parabolic subgroups of type~\ref{paratype:b}.
  Notice that these parabolic subgroups are isomorphic to \(C_2\) by Lemma~\ref{lem:paratype} and two such groups are hence conjugate if their generators are.
  If there is no entry in the column \(\langle (\zeta_d I_2)s\rangle\), then there is only one conjugacy class of type~\ref{paratype:b}.
\end{remark}

\begin{table}
  \caption{Conjugacy classes of non-trivial parabolic subgroups, see Remark~\ref{rem:tabexplain} for the notation.}
  \label{tab:paras}
  \def\arraystretch{1.2}
  \begin{tabular}{c|c|c||c|c|c|c|c||c|c}
  \(H_0\)    & \(|{H_0}|\) &  minimal \(d\)  & \(C_2'\) & \(C_2''\) & \(C_3\) & \(C_4\) & \(C_5\) & \(\langle s\rangle\) & \(\langle (\zeta_dI_2) s\rangle\) \\
  \hline
  \(G_5\)    & 72 &  6               &          &           & 8       &         &         & \(d/2\)      & \(d/2\) \\
  \(G_7\)    & 144 &  12              & 6        &           & 8       &         &         & \(d/2\)      & \(d/2\) \\
  \(G_8\)    & 96 &  4               &          &           &         & 6       &         & \(d\)                &                 \\
  \(G_9\)    & 192 &  8               & 12       &           &         & 6       &         & \(d/2\)      & \(d/2\) \\
  \(G_{10}\) & 288 &  12              &          &           & 8       & 6       &         & \(d\)                &                 \\
  \(G_{11}\) & 576 &  24              & 12       &           & 8       & 6       &         & \(d/2\)      & \(d/2\) \\
  \(G_{12}\) & 48 &  10              & 12       &           &         &         &         & \(d\)                &                 \\
  \(G_{13}\) & 96 &  20              & 12       & 6         &         &         &         & \(d/2\)      & \(d/2\) \\
  \(G_{14}\) & 144 &  6               & 12       &           & 8       &         &         & \(d\)                &                 \\
  \(G_{15}\) & 288 &  12              & 12       & 6         & 8       &         &         & \(d/2\)      & \(d/2\) \\
  \(G_{16}\) & 600 &  10              &          &           &         &         & 12      & \(d/2\)      & \(d/2\) \\
  \(G_{17}\) & 1200 &  20              & 30       &           &         &         & 12      & \(d/2\)      & \(d/2\) \\
  \(G_{18}\) & 1800 &  30              &          &           & 20      &         & 12      & \(d/2\)      & \(d/2\) \\
  \(G_{19}\) & 3600 &  60              & 30       &           & 20      &         & 12      & \(d/2\)      & \(d/2\) \\
  \(G_{20}\) & 360 &  6               &          &           & 20      &         &         & \(d/2\)      & \(d/2\) \\
  \(G_{21}\) & 720 &  12              & 30       &           & 20      &         &         & \(d/2\)      & \(d/2\) \\
  \(G_{22}\) & 240 &  8               & 30       &           &         &         &         & \(d/2\)      & \(d/2\) \\
  \end{tabular}
\end{table}

\subsection{Normalizers of parabolic subgroups}

We turn to the question whether a parabolic subgroup \(P\leq G\) has a complement in its normalizer \(N_G(P)\).
For this, we first consider parabolic subgroups of type~\ref{paratype:a} and extend the result from \cite{MT18} to the normalizer in \(H\).

\begin{lemma}
  \label{lem:H0toH}
  Let \(P\leq H_0\) be a parabolic subgroup.
  Then there is a subgroup \(C\leq N_H(P)\) such that \(N_H(P) = P\rtimes C\).
  That is, \(P\) has a complement in its normalizer in \(H\).
\end{lemma}
\begin{proof}
  Assume first that \(P\) is not one of the groups of order 2 in a conjugacy class of length 6 for \(H_0\in\{G_{13},G_{15}\}\).
  Hence \(N_{H_0}(P) = P\times Z(H_0)\) by \cite[Thm.~5.5]{MT18}.
  We have \(|{P^H}| = |{P^{H_0}}|\) for the lengths of the conjugacy classes because \(H = Z(H) H_0\).
  Further, \(Z(H)\leq N_H(P)\), so \(N_H(P) = P\times Z(H)\) in \(H\).

  Let now \(P\) be one of the excluded groups.
  We again obtain a complement of \(P\) in \(N_H(P)\) by extending the complement of \(P\) in \(N_{H_0}(P)\) given in \cite[Lem.~5.14]{MT18}.
  Let \(k\in\BBZ_{\geq 1}\) be the order of the image of \(\det :N_H(P)\to \BBC^\times\) and set \(C = \{h\in N_H(P)\mid \det(h)^{k/2} = 1\}\).
  As in the proof of \cite[Lem.~5.14]{MT18}, it follows that \([N_H(P):C] = 2\) and \(C\cap P = 1\) if and only if \(k/2\) is not divisible by 2.
  Checking the table in \cite[(3.6)]{Coh76}, we see that \(H = \mu_l\mathsf O\) with \(l\in \BBZ\) a multiple of 4.
  Further, \(H_0\in\{G_{13}, G_{15}\}\) if and only if \(l\) is not divisible by 8.
  Because \(\mathsf O\leq\SL_2(\BBC)\), we see that \(k\mid \frac{l}{2}\), so \(k\) is not divisible by 4.
  Hence \(C\) is the desired complement of \(P\) in \(N_H(P)\).
\end{proof}

\begin{lemma}
  \label{lem:gscomp}
  Let \(r\in H\) be a complex reflection and \(g\in H\) with \(gs\in N_G(\langle r\rangle)\).
  Then \((gs)^2 = -I_2\).
\end{lemma}
\begin{proof}
  This is an elementary computation.
  By switching to an appropriate basis, we may assume that \(r = \begin{psmallmatrix} 1 & \\ & \zeta\end{psmallmatrix}\) for a root of unity \(1\neq \zeta\in \BBC^\times\).
  We have \(srs = \begin{psmallmatrix} \zeta^{-1} & \\ & 1\end{psmallmatrix}\) by Lemma~\ref{lem:sconj}.
  Hence \(\det(gs rsg^{-1}) = \det(r)^{-1}\) and from \(gsrsg^{-1}\in \langle r\rangle\) we obtain \(gs rsg^{-1} = r^{-1}\).
  Write \(g = \begin{psmallmatrix} a & b\\ c & d\end{psmallmatrix}\).
  Because \(g\in\GU_2(\BBC)\), we have \(g^{-1} = \begin{psmallmatrix} \overline a & \overline c \\ \overline b & \overline d\end{psmallmatrix}\).
  Using the identities one derives from \(gg^{-1} = I_2\) (that is, \(a\overline a  + b\overline b = 1\) etc.), we compute \[gsrsg^{-1} = \begin{pmatrix} (\zeta^{-1} - 1)a\overline a + 1 & (\zeta^{-1} - 1)a\overline c \\ (\zeta^{-1} - 1)c\overline a & (\zeta^{-1} - 1) c\overline c + 1\end{pmatrix}.\]
  Comparing with \(r^{-1}\), we derive \(a = 0\) and \(c\overline c = 1\) and from this \(d = 0\) and \(b\overline b = 1\).
  So, \(g = \begin{psmallmatrix} & b\\ c & \end{psmallmatrix}\) with \(b, c\in \BBC^\times\) are roots of unity.
  One computes \((gs)^2 = -I_2\) again using Lemma~\ref{lem:sconj} and this is independent of the chosen basis, because \(-I_2\in Z(\GU_2(\BBC))\).
\end{proof}

\begin{prop}
  \label{prop:compp1}
  Let \(P\leq G\) be a parabolic subgroup of type~\ref{paratype:a}.
  Then \(P\) has a complement in the normalizer \(N_G(P)\).
\end{prop}
\begin{proof}
  Because \(P\) is of type~\ref{paratype:a}, \(P\) is a parabolic subgroup of \(H_0\) and hence \(P\) has a complement \(C\) in \(N_H(P)\) by Lemma~\ref{lem:H0toH}.

  If \(|{P^G}| > |{P^H}|\), then we must have \(|{N_H(P)}| = |{N_G(P)}|\) because \(|{G}| = 2|{H}|\).
  So \(N_G(P) = N_H(P)\) and the complement \(C\) of \(P\) in \(N_H(P)\) is also one of \(P\) in \(N_G(P)\).

  Otherwise \(|{P^G}| = |{P^H}|\) and we have \(|{N_G(P)}| = 2|{N_H(P)}|\).
  Clearly, \(N_G(P)\setminus N_H(P)\subseteq Hs\).
  Let \(gs\in N_G(P)\setminus N_H(P)\).
  We claim that \(C' = \langle C, gs\rangle\) is a group of order \(|{C'}| = 2|{C}|\).
  Indeed, the group \(P\) is cyclic and generated by a reflection \(r\in P\).
  By Lemma~\ref{lem:gscomp}, we hence have \((gs)^2 = -I_2\) and so \((gs)^2\in C\) by the explicit constructions in Lemma~\ref{lem:H0toH}.
  Further, if \(h_1gs, h_2gs\in C'\) with \(h_1,h_2\in C\), then \(h_1gsh_2gs\in H\).
  If \(C = Z(H) = \mu_d I_2\), then \[h_1gsh_2gs = h_1\overline{\det(h_2)}h_2gsgs = -\overline{\det(h_2)}h_1h_2 \in C.\]
  In the other case, we conclude via \(\det(h_1gsh_2gs) = \frac{\det(h_1)}{\det(h_2)}\) that \(h_1gsh_2gs\in C\).
  Therefore, every element in \(C'\) can be written as \(h(gs)^l\) with \(h\in C\) and \(l\in\{0, 1\}\), so \(|{C'}| = 2|{C}|\).
  From this we conclude that \(C'\cap H = C\) and hence \(C'\cap P = \{1\}\).
  So \(N_G(P) = P\rtimes C'\) and \(C'\) is a complement of \(P\) in \(N_G(P)\).
\end{proof}

\begin{prop}
  \label{prop:compp2}
  Let \(P\leq G\) be a parabolic subgroup of type~\ref{paratype:b} and let \(H_1\leq H\) be the subgroup consisting of all elements of determinant 1.
  Then the normalizer of \(P\) is \(N_G(P) = P\rtimes H_1\).
  In particular, \(P\) has a complement in its normalizer.
\end{prop}
\begin{proof}
  Let \((\zeta I_2) s\in P\) be the non-trivial element of \(P\) with a root of unity \(\zeta\in \mu_d\).
  For \(g\in G\), we have either \(g\in H\) or \(gs\in H\).
  Because \((\zeta I_2) s\in N_G(P)\), we may assume without loss of generality that \(g\in H\).
  Lemma~\ref{lem:sconj} yields \[g(\zeta I_2)sg^{-1} = (\zeta I_2)s\overline{\det(g)}gg^{-1} = (\zeta I_2) s\overline{\det(g)}.\]
  Hence \(g\in N_G(P)\cap H\) if and only if \(g\in H_1\).
  Because \(G = H\times \langle s\rangle\), it follows immediately that \(N_G(P) = P\rtimes H_1\).
\end{proof}

\begin{remark}
  \label{rem:possh1}
  The group \(H_1\) in Proposition~\ref{prop:compp2} only depends on the complex reflection group \(H_0\).
  There are only three possibilities for \(H_1\), namely \(\mathsf T, \mathsf O,\mathsf I\leq\SL_2(\BBC)\).
  The precise group can be derived from \(H_0\) with the table in \cite[(3.6)]{Coh76}.
\end{remark}

\begin{remark}
  In \cite[Lem.~5.1]{MT18}, it is proved that for a rank 1 parabolic subgroup of a complex reflection group, the normalizer coincides with the centralizer.
  This is false in general for quaternionic reflection groups; the proof in \cite{MT18} does not work over a skew field.
  For example, let \(P = \langle r\rangle\) be a parabolic subgroup of isomorphism type \(C_4\) in \(G = E(G_9)\).
  Write \(r^G\) for the conjugacy class of \(r\) in \(G\).
  We have \(|{G}| = 2|{G_9}|\) and \(|{r^G}| = 2|{r^H}|\) by Lemma~\ref{lem:reflconj}.
  So, \(|{C_G(P)}| = |{C_H(P)}|\).
  On the other hand, \(|{P^G}| = 6 = |{P^H}|\) and this implies \(|{N_G(P)}| = 2|{N_H(P)}|\).
  We have \(N_H(P) = C_H(P)\) by \cite[Lem.~5.1]{MT18} and so \(N_G(P) \neq C_G(P)\).
\end{remark}

\section{Primitive groups with primitive complexification}
\label{sec:primprim}

There are 13 irreducible quaternionic reflection groups of rank at least 2 which are primitive and have a primitive complexification, see \cite[Thm.~4.2]{Coh80}.
These groups are denoted by \(W(X)\), where \(X\) is a quaternionic root system.
The root systems are labelled by the capital letters \(O\) to \(U\) with some indices.

We list the parabolic subgroups of these groups in the tables in Section~\ref{sec:tabsprim} and give a complement in the normalizer, if such a complement exists.
We may summarize the existence of such a complement as follows.

\begin{prop}
  \label{prop:primprim}
  Let \(G\) be a primitive, irreducible quaternionic reflection group with primitive complexification of rank at least 2 and let \(P\leq G\) be a parabolic subgroup.
  Then \(P\) has a complement in its normalizer \(N_G(P)\) if and only if the pair \((G, P)\) is not one of the following:
\[(W(Q), G(4, 2, 2)),\ (W(R), \mathsf C_2),\ (W(S_3), \mathsf C_2),\ (W(U), \mathsf C_2\times \mathsf C_2),\ (W(U), \mathsf C_2\times\mathsf C_2\times\mathsf C_2).\]
\end{prop}

\subsection{An example}

We study one of the parabolic subgroups that does not admit a complement in its normalizer in detail.
Additionally, we provide a computer-free argument showing that there is no such complement.
\begin{example}
  \textbf{Construction of the group.}
  Let \(G = W(Q)\).
  All reflections in \(G\) are of order 2 and we may describe them via the quaternionic lines that are orthogonal to the corresponding reflection hyperplanes.
  These root lines are given in \cite[Table~II]{Coh80} as follows.
  Let \(V = \BBH^3\).
  We have the vectors \((1, 0, 0)\), \((0, 1, 0)\) and \((0, 0, 1)\) as well as the 12 vectors
  \[(1, \pm1, 0), (1, \pm\bfi, 0), (1, 0, \pm1), (1, 0, \pm\bfi), (0, 1, \pm1), (0, 1, \pm\bfi).\]
  The corresponding reflections generate the reflection group \(G(4, 2, 3)\) acting on \(V\).

  Let \(\alpha = \frac{1}{2}(1 - \bfi - \bfj - \sqrt 5 \bfk)\).
  Then the root system \(Q\) additionally contains the 48 vectors of the form \((a, b, c \alpha)\) and all permutations of the coordinates of this with \(a, b, c\in \langle\bfi\rangle\) and \(abc = 1\).
  We have \(|{Q}| = 63\) in total, the reflections corresponding to the roots in \(Q\) generate the group \(G = W(Q)\) and every reflection in \(G\) stems from one of the listed roots, see \cite[Rem.~4.3~(ii)]{Coh80}.
  The order of \(G\) is \(|{G}| = 12096\), see \cite[Table~III]{Coh80}.

  \textbf{The normalizer.}
  Let \(P\leq G\) be the stabilizer of the vector space spanned by \((1, 0, 0)\).
  Then \(P\) is a reflection group \cite{BST23} and contains the reflections with roots \[(0, 1, 0), (0, 0, 1), (0, 1, 1), (0, 1, \bfi), (0, 1, -1), (0, 1, -\bfi).\]
  That is, \(P\) is isomorphic (as a reflection group) to the group \(G(4, 2, 2)\) acting on the space \(\langle (0, 1, 0), (0, 0, 1)\rangle\) and \(|{P}| = 16\).
  For any two root lines \(v_1\) and \(v_2\), there is an element \(g\in G\) with \(g.v_1 = v_2\).
  Hence the stabilizer of any root line is conjugate to \(P\) in \(G\).
  Therefore the conjugacy class of \(P\) contains at least 63 groups and we obtain \(|{N_G(P)}| \leq |{G}| / 63 = 192\) as upper bound for the order of the normalizer of \(P\) in \(G\).
  We have the inclusions \[P\leq G(4, 2, 3) \leq G\] and hence \(P\leq N_{G(4, 2, 3)}(P) \leq N_G(P)\).
  Write \[N' := N_{G(4, 2, 3)}(P)\text{ and }N := N_G(P).\]
  The elements of \(N'\) are those matrices in \(G(4, 2, 3)\) that leave the block structure given by \(\Fix(P) \oplus \Fix(P)^\perp = \langle (1, 0, 0)\rangle\oplus \langle (0, 1, 0), (0, 0, 1)\rangle\) invariant, so we have \(|{N'}| = 64\).

  The group \(N\) is strictly larger than \(N'\).
  Indeed, consider \[g = \frac{1}{2}\begin{pmatrix} -2 + \alpha + \alpha \bfi & 0 & 0 \\ 0 & -\bfi - 1 & \bfi + 1\\ 0 & \bfi - 1 & \bfi - 1 \end{pmatrix}.\]
    We have \(g\in G\) because \(g\) is the product of the reflections corresponding to the root lines \((1, \bfi \alpha, -\bfi)\), \((\bfi, \bfi, \alpha)\), \((\bfi, -\bfi, \alpha)\), \((\bfi, 1, \bfi \alpha)\).
  Clearly, \(g\in N\) because \(g\) leaves the decomposition \(\Fix(P)\oplus \Fix(P)^\perp\) invariant.
  Further, one computes \(\ord_G(g) = 3\), so \(g\notin N'\).
  In particular, we have \(|{\langle N', g\rangle}| \geq 192\), so \(N = \langle N', g\rangle\) and \(|{N}| = 192\).

  \textbf{Elements of order 3.}
  In order to show that \(P\) does not have a complement in \(N\), we require more information on the elements of order 3 in \(N\).
  Because \(3\nmid |{N'}|\), we have \(N = \langle g\rangle N'\) with \(g\) as above.
  Let \(\omega = g_{1, 1}\), the \((1, 1)\)-entry of the matrix \(g\), and let \(h\in N\).
  Then there is an \(n\in N'\) and a \(k\in\{0, 1, 2\}\) with \(h = g^k n\), so \(h_{1, 1} = \omega^kn_{1, 1}\) with \(n_{1, 1}\in \langle \bfi\rangle\).
  Assume \(\ord_G(h) = 3\), hence in particular \(h_{1, 1}^3 = 1\).
  We cannot have \(k = 0\), because then \(h\in N'\) and \(3\nmid |{N'}|\).
  One checks that \(\ord_{\BBH^\times}(-\omega^k) = 6\) and \(\ord_{\BBH^\times}(\omega^k(\pm \bfi)) = 4\) for \(k = 1, 2\).
  We conclude \(h_{1, 1} = \omega^k\) and \(n\in P\).
  Therefore, any element in \(N\) of order 3 is of the form \(gp\) or \(g^2p\) for some \(p\in P\).

  \textbf{Non-existence of a complement.}
  Assume there is a complement \(C\leq N\) of \(P\) in \(N\).
  Then \(C' \coloneqq C\cap N'\) is a complement of \(P\) in \(N'\).
  By \cite[Thm.~3.12]{MT18}, the group \(C'\) acts as reflection group \(\mathsf C_4\) on \(\Fix(P)\).
  Let \(h\in C'\) be a generator of \(C'\).
  We may assume without loss of generality that \(h_{1, 1} = \bfi\) and so the possibilities for \(h\) are of the form \[\begin{psmallmatrix} \bfi & & \\ & a & \\ & & b\end{psmallmatrix}\text{ or }\begin{psmallmatrix}\bfi & & \\ & & a\\ b & & \end{psmallmatrix}\] with \(a, b\in\langle \bfi\rangle\) and \(ab = \pm \bfi\).
  The matrices of the second (non-diagonal) form are of order 8, so \(h\) must be a diagonal matrix.
  In particular, the square of \(h\) is \(\diag(-1, -1, 1)\) or \(\diag(-1, 1, -1)\).
  We assume without loss of generality that \[h^2 = \diag(-1, -1, 1) \in C.\]

  By construction, we have \(|{C}| = |{N}|/|{P}| = 12\) and \(C\) must contain an element \(\tilde g\in N\) of order 3.
  Recall from above that \(\tilde g_{1, 1} \in\{\omega, \omega^2\}\) and assume without loss of generality that \(\tilde g_{1, 1} = \omega\).
  The element \(\tilde gh^2\) must have order divisible by 6, because \((\tilde gh^2)_{1, 1} = -\omega\).

  Assume that \(\ord_G(\tilde gh^2) = 6\).
  We show that \(\langle \tilde g\rangle\cap\langle \tilde gh^2\rangle = \{1\}\).
  Assume \(\tilde g^k = (\tilde g h^2)^l\) for some choice of \(k = 1, 2\) and \(l = 1,\dots, 5\).
  From the entry \[(\tilde g^k)_{1, 1} = \omega^k = (-\omega)^l = ((\tilde gh^2)^l)_{1, 1},\] we see that the only permissible options are \((k, l)\in\{(1, 4), (2, 2)\}\).
  Both choices of \(k\) and \(l\) give the same relation \(\tilde g^2 = (\tilde g h^2)^2\) (after squaring in the case \(k = 1\)) and this results in \(\tilde g = h^2\tilde gh^2\).
  Because \(\tilde g\in N\), there are \(a, b, c, d\in \BBH\) with \(g = \diag(\omega,\begin{psmallmatrix} a & b \\ c & d\end{psmallmatrix})\) and one computes \(h^2\tilde gh^2 = \diag(\omega, \begin{psmallmatrix} a & -b \\ -c & d\end{psmallmatrix})\).
  We have \(b\neq 0\neq c\) because \(\tilde g = gp\) for some \(p\in P\), so the above choices for \(k\) and \(l\) are in fact not permissible either.
  Hence we have \(\langle \tilde g\rangle\cap\langle \tilde gh^2\rangle = \{1\}\) and in particular \(|{C}| \geq 3\cdot 6 = 18\), a contradiction.

  Finally, assume that \(\ord_G(\tilde gh^2) = 12\).
  In particular, the group \(C\) is cyclic and we must have \(h = (\tilde gh^2)^l\) for some \(0\leq l\leq 11\).
  But this is not possible since \(((\tilde gh^2)^l)_{1, 1} \neq \bfi\) for all \(l\).
  In conclusion, \(P\) does not have a complement in \(N\).
\end{example}

\subsection{Tables of parabolic subgroups}
\label{sec:tabsprim}

The non-trivial parabolic subgroups and their complements for the primitive quaternionic reflection groups \(G\) with primitive complexification are listed in Tables \ref{tab:wop}, \ref{tab:wqr}, \ref{tab:ws}, \ref{tab:wt} and \ref{tab:wu}.
All data was computed using the computer algebra system OSCAR \cite{Osc26, Dec+25}.
The code to reproduce the results in the tables can be found at \url{https://gitlab.com/math5724907/normalizersofparabolics} together with precomputed parabolic subgroups that can be loaded in OSCAR.

The columns of the tables contain the following information.
\begin{itemize}
  \item \(P\): a parabolic subgroup of \(G\) with pointwise fixed space \(U\).
  \item \(|{P^G}|\): the length of the conjugacy class of \(P\) in \(G\).
  \item rank: the rank of \(P\), that is, the dimension of the complement of \(U\).
  \item \(Q\): the parabolic subgroup of \(G\) fixing the orthogonal complement \(U^\perp\).
  \item \(C\): a complement of \(P\) in \(N_G(P)\) (described in relation to \(C^\circ\) or \(Q\)).
  \item \(C^\circ\): the largest subgroup of \(C\) acting on \(U\) by reflections.
\end{itemize}

The groups are written in the various notations established in earlier sections.
To summarize:
\begin{itemize}
  \item \(\mathsf C_d\), \(\mathsf D_d\), \(\mathsf T\), \(\mathsf O\) and \(\mathsf I\) are the Kleinian groups as in Notation~\ref{not:kleinian}.
  \item \(G(m, p, n)\) and \(G_k\) denote complex reflection groups following \cite{ST54}.
  \item \(\mathfrak S_n\) denotes the symmetric group acting irreducibly on an \((n - 1)\)-dimensional space.
  \item \(G_n(K, H)\) denotes an imprimitive group as described in Section~\ref{sec:imprim}.
  \item \(G(K, H, \phi)\) denotes an `exceptional' imprimitive group of rank 2, see Section~\ref{sec:imprimexcep}.
  \item \(E(H)\) denotes a primitive group with imprimitive complexification as described in Section~\ref{sec:primimprim}.
  \item \(W(X)\) denotes a primitive group with primitive complexification corresponding to the root system \(X\), see \cite[Sect.~4]{Coh80}.
\end{itemize}

\begin{table}[b]
  \caption{Non-trivial parabolic subgroups of the rank 2 primitive groups \(G\).}
  \label{tab:wop}
  \def\arraystretch{1.2}
  \begin{tabular}{c|ccc|ccc}
    \(G\) & \(P\) & \(|{P^G}|\) & rank & \(Q\) & \(C\) & \(C^\circ\) \\
    \hline
    \(W(O_1)\) & \(\mathsf C_3\) & 10 & 1 &\(1\) & \(C^\circ\) & \(\mathsf C_4\) \\
    \hline
    \(W(O_2)\) & \(\mathsf C_3\) & 20 & 1 &\(\mathsf C_3\) & \(C^\circ\) & \(\mathsf D_3\) \\
    \hline
    \(W(O_3)\) & \(\mathsf C_2\) & 30 & 1 & \(\mathsf C_2\) & \(C^\circ\) & \(\mathsf T\) \\
               & \(\mathsf C_3\) & 20 & 1 & \(\mathsf C_3\) & \(C^\circ\) & \(\mathsf D_6\) \\
    \hline
    \(W(P_1)\) & \(\mathsf C_4\) & 10 & 1 & \(\mathsf C_4\) & \(C^\circ\) & \(\mathsf D_2\) \\
    \hline
    \(W(P_2)\) & \(\mathsf D_2\) & 10 & 1 & \(\mathsf D_2\) & \(C^\circ\) & \(\mathsf T\) \\
    \hline
    \(W(P_3)\) & \(\mathsf C_2\) & 40 & 1 & \(\mathsf C_2\) & \(C^\circ\) & \(\mathsf O\) \\
               & \(\mathsf D_2\) & 10 & 1 & \(\mathsf D_2\) & \(C^\circ\) & \(\mathsf O\)
  \end{tabular}
\end{table}

\begin{table}[b]
  \caption{Non-trivial parabolic subgroups of the rank 3 primitive groups \(G\).}
  \label{tab:wqr}
  \def\arraystretch{1.2}
  \begin{tabular}{c|ccc|ccc}
    \(G\) & \(P\) & \(|{P^G}|\) & rank & \(Q\) & \(C\) & \(C^\circ\) \\
    \hline
    \(W(Q)\) & \(\mathsf C_2\) & 63 & 1 & \(G(4, 2, 2)\) & \(C^\circ\) & \(G_8\) \\
     & \(G(3, 3, 2)\) & 336 & 2 & 1 & \(C^\circ\) & \(\mathsf C_6\) \\
     & \(G(4, 2, 2)\) & 63 & 2 & \(\mathsf C_2\) & (none) & (none) \\
    \hline
    \(W(R)\) & \(\mathsf C_2\) & 315 & 1 & \(G_2(\mathsf D_2, \mathsf C_2)\) & (none) & (none) \\
     & \(G(3, 3, 2)\) & 8400 & 2 & 1 & \(C^\circ\) & \(\mathsf T\) \\
     & \(G(5, 5, 2)\) & 1008 & 2 & 1 & \(C^\circ\) & \(\mathsf I\) \\
     & \(G_2(\mathsf D_2, \mathsf C_2)\) & 315 & 2 & \(\mathsf C_2\) & \(C^\circ\) & \(\mathsf I\)
  \end{tabular}
\end{table}

\begin{table}
  \caption{Non-trivial parabolic subgroups of the groups \(W(S_i)\).}
  \label{tab:ws}
  \def\arraystretch{1.2}
  \begin{tabular}{c|ccc|ccc}
    \(G\) & \(P\) & \(|{P^G}|\) & rank & \(Q\) & \(C\) & \(C^\circ\) \\
    \hline
    \(W(S_1)\) & \(\mathsf C_2\) & 36 & 1 & \(\mathsf C_2^3\) & Extension of \(A_4\) by \(Q\) & \(Q\) \\
    & \(\mathsf C_2^2\) & 54 & 2 & \(\mathsf C_2^2\) & \(C^\circ\) & \(G_2(\mathsf D_2, \mathsf C_2)\) \\
    & \(G(3, 3, 2)\) & 48 & 2 & 1 & \(C^\circ\) & \(G_4\) \\
    & \(G(3, 3, 2)\) & 48 & 2 & 1 & \(C^\circ\) & \(G_4\) \\
    & \(G(3, 3, 2)\) & 48 & 2 & 1 & \(C^\circ\) & \(G_4\) \\
    & \(G(3, 3, 2)\) & 48 & 2 & 1 & \(C^\circ\) & \(G_4\) \\
    & \(\mathsf C_2^3\) & 36 & 3 & \(\mathsf C_2\) & \(C^\circ\) & \(\mathsf T\) \\
    & \(G(2, 2, 3)\) & 36 & 3 & 1 & \(C^\circ\) & \(\mathsf D_2\)\\
    & \(G(2, 2, 3)\) & 36 & 3 & 1 & \(C^\circ\) & \(\mathsf D_2\)\\
    & \(G(2, 2, 3)\) & 36 & 3 & 1 & \(C^\circ\) & \(\mathsf D_2\)\\
    & \(G(2, 2, 3)\) & 36 & 3 & 1 & \(C^\circ\) & \(\mathsf D_2\)\\
    & \(G(3, 3, 3)\) & 64 & 3 & 1 & \(C^\circ\) & \(\mathsf C_2\)\\
    \hline
    \(W(S_2)\) & \(\mathsf C_2\) & 72 & 1 & \(G(2, 1, 3)\) & Extension of \(A_4\) by \(Q\) & \(Q\) \\
     & \(\mathsf C_2^2\) & 216 & 2 & \(\mathsf C_2^2\) & \(C^\circ\) & \(G(\mathsf T, \mathsf C_2, \rho(\gamma))\)\\
     & \(G(3, 3, 2)\) & 96 & 2 & \(G(3, 3, 2)\) & \(C^\circ\) & \(E(G_5)\)\\
     & \(G(3, 3, 2)\) & 576 & 2 & 1 & \(C^\circ\) & \(G_4\) \\
     & \(G(2, 1, 2)\) & 54 & 2 & \(G(2, 1, 2)\) & \(C^\circ\) & \(E(G_8)\) \\
     & \(\mathsf C_2 \times G(3, 3, 2)\) & 288 & 3 & 1 & \(C^\circ\) & \(\mathsf T\) \\
     & \(G(2, 2, 3)\) & 864 & 3 & 1 & \(C^\circ\) & \(\mathsf C_4\) \\
     & \(G(2, 1, 3)\) & 72 & 3 & \(\mathsf C_2\) & \(C^\circ\) & \(\mathsf T\) \\
     & \(G(3, 3, 3)\) & 256 & 3 & 1 & \(C^\circ\) & \(\mathsf C_6\) \\
     & \(G(4, 4, 3)\) & 108 & 3 & 1 & \(C^\circ\) & \(\mathsf D_2\) \\
    \hline
    \(W(S_3)\) & \(\mathsf C_2\) & 180 & 1 & \(G_3(\mathsf D_2, \mathsf C_2)\) & (none)  & (none) \\
    & \(\mathsf C_2^2\) & 2160 & 2 & \(\mathsf C_2^2\) & \(C^\circ\) & \(G_2(\mathsf T, \mathsf D_2)\) \\
    & \(G(3, 3, 2)\) & 3840 & 2 & \(G(3, 3, 2)\) & \(C^\circ\) & \(E(G_5)\) \\
    & \(G_2(\mathsf D_2, \mathsf C_2)\) & 54 & 2 & \(G_2(\mathsf D_2, \mathsf C_2)\) & \(C^\circ\) & \(W(P_2)\) \\
    & \(\mathsf C_2\times G(3, 3, 2)\) & 11520 & 3 & 1 & \(C^\circ\) & \(\mathsf T\) \\
    & \(G(2, 2, 3)\) & 17280 & 3 & 1 & \(C^\circ\) & \(\mathsf D_2\) \\
    & \(G(3, 3, 3)\) & 2560 & 3 & 1 & \(C^\circ\) & \(\mathsf T\) \\
    & \(G_3(\mathsf D_2, \mathsf C_2)\) & 180 & 3 & \(\mathsf C_2\) & \(C^\circ\) & \(\mathsf T\)
  \end{tabular}
\end{table}

\begin{table}
  \caption{Non-trivial parabolic subgroups of \(W(T)\).}
  \label{tab:wt}
  \def\arraystretch{1.2}
  \begin{tabular}{c|ccc|ccc}
    \(G\) & \(P\) & \(|{P^G}|\) & rank & \(Q\) & \(C\) & \(C^\circ\) \\
    \hline
    \(W(T)\) & \(\mathsf C_2\) & 180 & 1 & \(G_{23}\) & Extension of \(A_5\) by \(Q\) & \(Q\) \\
    & \(\mathsf C_2^2\) & 1350 & 2 & \(\mathsf C_2^2\) & \(C^\circ\) & \(G(\mathsf I, \mathsf C_2, \id)\)\\
    & \(G(3, 3, 2)\) & 3600 & 2 & 1 & \(C^\circ\) & \(W(O_1)\) \\
    & \(G(3, 3, 2)\) & 600 & 2 & \(G(3, 3, 2)\) & \(C^\circ\) & \(E(G_{20})\) \\
    & \(G(5, 5, 2)\) & 216 & 2 & \(G(5, 5, 2)\) & \(C^\circ\) & \(E(G_{16})\)  \\
    & \(\mathsf C_2\times G(3, 3, 2)\) & 1800 & 3 & 1 & \(C^\circ\) & \(\mathsf I\) \\
    & \(\mathsf C_2\times G(5, 5, 2)\) & 1080 & 3 & 1 & \(C^\circ\) & \(\mathsf I\) \\
    & \(G(2, 2, 3)\) & 13500 & 3 & 1 & \(C^\circ\) & \(\mathsf D_2\)  \\
    & \(G(2, 2, 3)\) & 900 & 3 & 1 & \(C^\circ\) & \(\mathsf I\)  \\
    & \(G(3, 3, 3)\) & 4000 & 3 & 1 & \(C^\circ\) & \(\mathsf D_3\)  \\
    & \(G_{23}\) & 180 & 3 & \(\mathsf C_2\) & \(C^\circ\) & \(\mathsf I\) \\
    & \(G(5, 5, 3)\) & 864 & 3 & 1 & \(C^\circ\) & \(\mathsf D_5\)
  \end{tabular}
\end{table}

\begin{table}
  \caption{Non-trivial parabolic subgroups of \(W(U)\).}
  \label{tab:wu}
  \def\arraystretch{1.2}
  \begin{tabular}{c|ccc|ccc}
    \(G\) & \(P\) & \(|{P^G}|\) & rank & \(Q\) & \(C\) & \(C^\circ\) \\
    \hline
    \(W(U)\) & \(\mathsf C_2\) & 165 & 1 & \(W(S_1)\) & Extension of \(A_4\) by \(Q\) & \(Q\) \\
    & \(\mathsf C_2^2\) & 2970 & 2 & \(\mathsf C_2^3\) & (none) & (none) \\
    & \(G(3, 3, 2)\) & 3520 & 2 & \(G(3, 3, 3)\) & \(C^\circ\) & \(G_{26}\) \\
    & \(\mathsf C_2^3\) & 2970 & 3 & \(\mathsf C_2^2\) & (none) & (none) \\
    & \(\mathsf C_2\times G(3, 3, 2)\) & 31680 & 3 & 1 & \(C^\circ\) & \(G_5\) \\
    & \(G(2, 2, 3)\) & 23760 & 3 & \(\mathsf C_2\) & Extension of \(C_3\) by \(C^\circ\) & \(\mathsf C_2\times \mathsf D_2\) \\
    & \(G(3, 3, 3)\) & 3520 & 3 & \(G(3, 3, 2)\) & \(C^\circ\) & \(E(G_5)\) \\
    & \(\mathsf C_2 \times G(2, 2, 3)\) & 23760 & 4 & 1 & \(C^\circ\) & \(\mathsf T\) \\
    & \(\mathsf C_2\times G(3, 3, 3)\) & 10560 & 4 & 1 & \(C^\circ\) & \(\mathsf T\) \\
    & \(\mathfrak S_5\) & 38016 & 4 & 1 & \(C^\circ\) & \(\mathsf C_6\) \\
    & \(G(3, 3, 4)\) & 7040 & 4 & 1 & \(C^\circ\) & \(\mathsf C_6\) \\
    & \(W(S_1)\) & 165 & 4 & \(\mathsf C_2\) & \(C^\circ\) & \(\mathsf T\)
  \end{tabular}
\end{table}


\clearpage

\begin{thebibliography}{aaaaaaa}

  \bibitem[Bea00]{Bea00}
    A.~Beauville, \emph{Symplectic singularities}, Invent.~Math. \textbf{139} (2000), no.~3, 541--549.

  \bibitem[Bel16]{Bel16}
    G.~Bellamy, \emph{Counting resolutions of symplectic quotient singularities}, Compos.~Math. \textbf{152} (2016), no.~1, 99--114.

  \bibitem[BRS26]{BRS26}
    G.~Bellamy, G.~Röhrle, J.~Schmitt,
    \emph{Namikawa--Weyl groups of symplectic quotient singularities}, 2026,  \url{http://arxiv.org/abs/2607.24158}.

  \bibitem[BST18]{BST18}
    G.~Bellamy, T.~Schedler, U.~Thiel, \emph{Hyperplane arrangements associated to symplectic quotient singularities}, Phenomenological approach to algebraic geometry, Banach Center Publ., vol.~116, Polish Acad. Sci. Inst. Math., Warsaw, 2018, pp. 25--45.

  \bibitem[BST23]{BST23}
    G.~Bellamy, J.~Schmitt, U.~Thiel, \emph{On parabolic subgroups of symplectic reflection groups}, Glasg.~Math.~J. \textbf{65} (2023), no.~2, 401--413.

  \bibitem[BH99]{BH99}
    B.~Brink and R.~B. Howlett, \emph{Normalizers of parabolic subgroups
      in {C}oxeter groups}, Invent. Math. \textbf{136} (1999), 323--351.

  \bibitem[Coh76]{Coh76}
    A.~M.~Cohen,
    \emph{Finite complex reflection groups},
    Ann. Sci. Éc. Norm. Supér. (4) \textbf{9} (1976), no.~3, 379--436.

  \bibitem[Coh80]{Coh80}
    A.~M.~Cohen,
    \emph{Finite quaternionic reflection groups},
    J. Algebra \textbf{64} (1980), no.~2, 293--324.

  \bibitem[Dec+25]{Dec+25}
    W.~Decker, C.~Eder, C.~Fieker, M.~Horn, M.~Joswig (eds.). \emph{The computer algebra system {OSCAR}: algorithms and examples.} 1st ed., vol. 32. Algorithms and Computation in Mathematics.
    Springer, 2025.

  \bibitem[DuV64]{DuV64}
    P. Du~Val, \emph{Homographies, quaternions and rotations}. Oxford Mathematical Monographs, Clarendon Press, Oxford, 1964.

  \bibitem[GRS25]{GRS25}
    L.~Giordani, G.~Röhrle, J.~Schmitt,
    \emph{Invariants in the cohomology of the complement of quaternionic reflection arrangements}, 2025,  \url{http://arxiv.org/abs/2510.27311}.

  \bibitem[How80]{How80}
    R.~B. Howlett, \emph{Normalizers of parabolic subgroups of reflection
    groups}, J. Lond. Math. Soc. (2) \textbf{21} (1980), 62--80.

  \bibitem[MT18]{MT18}
    K.~Muraleedaran, D.~E.~Taylor,
    \emph{Normalisers of parabolic subgroups in finite unitary reflection groups}, J. Algebra \textbf{504} (2018), 479--505.

  \bibitem[Nam10]{Nam10}
    Y.~Namikawa, \emph{Poisson deformations of affine symplectic varieties, II}, Kyoto J.~Math. \textbf{50} (2010), 727--752.

  \bibitem[Osc26]{Osc26}
    \emph{OSCAR -- Open Source Computer Algebra Research system}. Version~1.7.0. The OSCAR Team, 2026, \url{https://www.oscar-system.org}.

  \bibitem[OS82]{OS82}
    P.~Orlik and L.~Solomon,
    \emph{Arrangements defined by unitary reflection groups},
    Math. Ann. \textbf{261} (1982), 339--357.

  \bibitem[OT92]{OT92}
    P.~Orlik and H.~Terao,
    \emph{Arrangements of hyperplanes},
    Springer-Verlag, 1992.

  \bibitem[ST54]{ST54}
    G.~C.~Shephard, J.~A.~Todd,
    \emph{Finite unitary reflection groups}, Canad.~J.~Math. \textbf{6} (1954), 274--304.

  \bibitem[Ste64]{Ste64}
    R.~Steinberg,
    \emph{Differential equations invariant under finite reflection groups},
    Trans. Amer. Math. Soc., \textbf{112} (1964), 392--400.

  \bibitem[Tay25]{Tay25}
    D.~E.~Taylor,
    \emph{Systems of imprimitivity of rank two quaternionic reflection groups}, 2025,
    \url{https://arxiv.org/abs/2510.22134}.

  \bibitem[Wal25]{Wal25}
    S.~Waldron,
    \emph{An elementary classification of the quaternionic reflection groups of rank two}, 2025,
    \url{https://arxiv.org/abs/2509.01849}.
\end{thebibliography}
\end{document}